\documentclass[11pt]{article}

\usepackage{tikz}
\usepackage{longtable}

\usepackage{amsmath,amsfonts,amsthm,authblk}
\usepackage{enumerate}
\usetikzlibrary{calc}
\theoremstyle{plain}
\newtheorem{theorem}{{\bf Theorem}} [section]
\newtheorem{lemma}[theorem]{{\bf Lemma}}
\newtheorem{prop}[theorem]{{\bf Proposition}}

\theoremstyle{definition}
\newtheorem{defn}[theorem]{{\bf Definition}}

\newcommand{\Kd}[1]{{\mathcal K}(#1)}

\usepackage{xhfill}
\newcommand{\ditto}[1][.4pt]{\xrfill{#1}~"~\xrfill{#1}}

\font\bbb=msbm10 scaled\magstep1

\newcommand{\FF}{\mbox{\bbb F}} 
 
\newcommand{\QQ}{\mbox{\bbb Q}} 
\newcommand{\ZZ}{\mbox{\bbb Z}} 

\font\bbb=msbm8 scaled\magstep1

\newcommand{\TPSSDS}{S^{\hspace{.2mm}d-1} \mbox{$\times
\hspace{-3mm}_{-}$} \, S^{\hspace{.1mm}1}}

\newcommand{\TPSST}{S^{\hspace{.2mm}2} \mbox{$\times
\hspace{-2.8mm}_{-}$} \, S^{\hspace{.1mm}1}}

\newcommand{\KBottle}{S^{\hspace{.2mm}1} \mbox{$\times
\hspace{-2.8mm}_{-}$} \, S^{\hspace{.1mm}1}}

 \newcommand{\TPPSS}{\kern.24em \rule width.08em height1.5ex
depth-.08ex \kern-.36em \times}

\newcommand{\mb}{\mathbb}
\newcommand{\tr}[1]{{\cal D}T_#1}
\newcommand{\ovr}[1]{\overrightarrow{#1}}

\author[1] {Benjamin A.~Burton}
\author[2] {Basudeb Datta}
\author[3] {Nitin Singh}
\author[4] {Jonathan Spreer}

\affil[1,4] {School of Mathematics and Physics, The University of Queensland,
Brisbane QLD 4072, Australia. \texttt{bab@maths.uq.edu.au}, \texttt{j.spreer@uq.edu.au}.}
\affil[2] {Department of Mathematics, Indian Institute of Science,
Bangalore 560\,012, India. \texttt{dattab@math.iisc.ernet.in}.}
\affil[3] {IBM India Research Lab, Bangalore 560\,045, India.
\texttt{nitisin1@in.ibm.com}.}

\begin{document}

\title{\bf A construction principle for tight and minimal triangulations of
manifolds}

\date{July 11, 2016}

\maketitle

\hrule

\bigskip

\noindent {\bf Abstract}

\smallskip
{\small Tight triangulations are exotic, but highly regular objects in combinatorial topology. A triangulation is tight if all its piecewise linear embeddings into a Euclidean space are as convex as allowed by the topology of the underlying manifold. Tight triangulations are conjectured to be strongly minimal, and proven to be so for dimensions $\leq 3$. However, in spite of substantial theoretical results about such triangulations, there are precious few examples. In fact, apart from dimension two, we do not know if there are infinitely many of them in any given dimension.

In this paper, we present a computer-friendly combinatorial scheme to obtain tight triangulations, and present new examples in dimensions three, four and
five. Furthermore, we describe a family of tight triangulated $d$-manifolds,
with $2^{d-1} \lfloor d / 2 \rfloor ! \lfloor (d-1) / 2 \rfloor !$
isomorphically distinct members for each dimension
$d \geq 2$. While we still do not know if there are infinitely many tight
triangulations in a fixed dimension $d > 2$, this result shows that there are
abundantly many.}

\bigskip

\noindent {\small {\em MSC 2010\,:} 57Q15, 57R05.

\noindent {\em Keywords:} (Embeddings of) abstract simplicial complexes;
Combinatorial manifold; Stacked sphere; Tight triangulation; Strongly minimal
triangulation.
}

\bigskip

\hrule

\section{Introduction}

Convexity is a very important concept in many mathematical fields and plays
a crucial role in countless ground breaking results. However, its use is
fundamentally limited to topological balls and spheres. Nonetheless there
exists a widespread intuitive notion that, given a family of topologically
equivalent (i.e., homeomorphic) spaces, some of them look more convex than
others.

Morse theory captures this intuition in a mathematically more precise way.
Critical points of a Morse function are essentially
places where a topological space (say, embedded in some Euclidean space) does
not look convex. The minimum number of such critical points is determined
by the topology of the topological space. A well-known lower bound of this number
is given by the sum of its Betti numbers. For many topological spaces,
including many standard examples of manifolds (the main focus of this article),
this bound is sharp. Hence, a fixed topological manifold $M$ usually only
requires a very small set of such critical points. In contrast, there exist
arbitrarily complicated distortions of $M$ where Morse functions exhibit an
arbitrarily high number of critical points.

Naturally, a representation of $M$ where a ``typical'' Morse function has
a small number of critical points is, for most applications, preferable over a representation where
Morse functions exhibit a large amount of critical points which are not
``topologically meaningful''.

\smallskip

This motivates the notion of {\em tightness}. Informally, an embedding of a
manifold into a Euclidean space is called {\em tight} if it is ``as convex as
possible'' given its topological properties, i.e., if it does not contain
unnecessary dents or bends. In Morse theoretic terms, tightness means that --
in a way yet to be defined -- all Morse functions of an embedded manifold
exhibit the minimum number of critical points.

In the smooth setting, an embedding of a manifold into a Euclidean space
is tight if intersecting it with a closed half space does not introduce any additional homology, see work by Alexandrov \cite{Alexandrov38ClassClosedSurf},
Milnor \cite{Milnor50RelBettiHypersurfIntGaussCurv}, Chern and
Lashof \cite{Chern57TotCurvImmMnf} and Kuiper \cite{Kuiper59ImmMinTotAbsCurv}.
In particular, balls are tight if and only if they are convex.

In the discrete world, a simplicial complex $X$ is
called {\em tight} if all simplex-wise linear embeddings of $X$ into a
Euclidean space are tight. For example, the three vertex triangulation of $S^1$ (a triangle) is tight, while the four vertex triangulation of $S^1$ (a four-cycle) is not -- it allows a non-tight embedding into the Euclidean plane
as a concave quadrilateral. A comprehensive treatment of the notion of {\em tight triangulations} (i.e., tight simplicial complexes) can be found in~\cite{PSM,Kuehnel99CensusTight}.

Geometrically, tight embeddings are optimal in the sense that they minimise
total absolute curvature. In the combinatorial setting, tight triangulations of manifolds in dimensions $\leq 3$ are known to be {\em strongly
minimal}, i.e., they contain the minimum number of simplices in each
dimension. This result is due to recent work by Bagchi, the second,
and the fourth author \cite{BDS2016}, but had been conjectured
a long time before in arbitrary dimensions by K\"uhnel and
Lutz~\cite{Kuehnel99CensusTight}.

This fact reveals a rare and deep link between the geometry and combinatorics of a (triangulable) topological space, making an analysis of this property
interesting for combinatorial topology as well as adjacent fields. In particular,  determining which manifolds admit tight triangulations and which do not is of central importance.

\begin{itemize}
	
\item
In dimension two, this question can be answered in full generality: An $n$-vertex triangulated surface $S$ is tight if and only if it contains the complete graph \cite{PSM}. Thus, its Euler characteristic must satisfy $\chi (S) = - \frac{1}{6} n (n-7)$. It then follows from work by Jungermann and Ringel \cite{Jungerman80MinTrigs}, and Ringel \cite{Ringel55NonOrientable} that for almost all integer solutions of the right hand side (with $n \geq 4$), the corresponding surfaces admit such an $n$-vertex tight triangulation. The only exception here must be made for the Klein bottle, which does not allow a $7$-vertex triangulation.
	
\item In dimension three, it is known that, apart from the $3$-sphere, at most
boundaries of handlebodies can be tight. Moreover, an $n$-vertex triangulation of a $3$-manifold $M$ is tight
if and only if the first Betti number of $M$ satisfies $\beta_1 (M,\mathbb{Z}_2) = \frac{1}{20}(n-4)(n - 5)$
\cite{BDS2016}. In particular, in dimension three tightness can be decided -- and tight triangulations of handlebodies
can be recognised -- in polynomial time \cite{BBDSS2016}. See Section~\ref{ssec:tntntrigs} for a brief overview of the recent results
leading to this characterisation of tight $3$-manifolds.

\item
In dimension four and beyond, our understanding of manifolds admitting tight triangulations is rather limited. See K\"uhnel~\cite{PSM} and more recent work by Bagchi and the second \& fourth authors~\cite{BDS,sp} for some progress in this direction.
\end{itemize}

Some of the difficulties in fully understanding the implications of tightness
stem from a simple fact. Despite a rich theory, there are very few examples of
tight triangulations for dimensions three and higher. In fact, this has been
repeatedly identified as a major gap of the field. We address this issue
in this paper.

More precisely, we present tight triangulations of several three-, four-, and five-manifolds, together with a general construction principle to find more examples. The method is an adaptation of the construction used by the second and third authors \cite{bdns}, which is more amenable to computer processing.
In addition, we present a set of $2^{d-1} \lfloor d / 2 \rfloor ! \lfloor (d-1) / 2 \rfloor !$ non-isomorphic tight triangulations of a $d$-manifold for every
$d\geq 2$.

The triangulations we obtain
are neighbourly members of the class ${\cal K}(d)$ of locally stacked triangulated manifolds, with each vertex link a
stacked sphere (and thus tight due to Effenberger~\cite{ef}). All our
triangulated manifolds are boundaries of  $(d+1)$-dimensional handlebodies.
The few previously known constructions of such manifolds include the boundary of
the $d$-simplex (the trivial genus zero $d$-handlebody), the family of sphere
bundles due to K\"uhnel~\cite{PSM} (including the famous unique 9-vertex
triangulation of $\TPSST$ of Walkup), and the infinite family constructed
by the second and third authors \cite{bdns}.

In dimension three, our results show that the existing necessary conditions on
the topology of a tight $3$-manifold triangulation (tight $3$-manifolds are
boundaries of handlebodies) are not too far from being sufficient.
For $n \in \{ 9,29,49,69,89,109 \}$ the non-orientable
version of the infinite family of boundaries of handlebodies admit $n$-vertex tight
triangulations. The number of triangulations we found for increasing values
of $n$ furthermore suggests that these new examples are part of an infinite
family of such triangulations.


In higher dimensions we are able to obtain examples similar to the
$3$-dimensional ones suggesting that boundaries of handlebodies satisfying
some necessary conditions for tightness in fact admit tight triangulations in
large numbers (see Theorem \ref{thm:family}).
These higher dimensional examples are also thought to be a first step
towards finding examples of tight triangulations with more diverse topologies
(i.e., examples of {\em non-stacked} tight manifolds).


The new examples of tight triangulations should be seen as a counterpart to
the theory of the field. While the latter typically provides
necessary conditions for manifolds to admit tight triangulations, our examples
provide sufficient conditions to enrich the theory and give new insight into
the link between the combinatorics and the topology of a triangulated manifold.

\subsection*{Acknowledgements}
The authors thank the anonymous referees for useful comments.
This work is supported by DIICCSRTE, Australia (project AISRF06660) \& DST, India (DST/INT/AUS/P-56/2013 (G)) under the Australia-India Strategic Research Fund.
The second author is also supported by the UGC Centre for
Advanced Studies.

\section{Preliminaries}

\subsection{Simplicial complexes and triangulated manifolds}

All simplicial complexes considered in this article are finite and abstract.
By a triangulated manifold, we mean an abstract
simplicial complex whose geometric carrier is a topological
manifold. We consider two complexes to be identical if they are
isomorphic. A $d$-dimensional simplicial complex is called {\em
pure} if all its maximal faces (called {\em facets}) are
$d$-dimensional. We identify a pure simplicial complex with the
set of facets in the complex. Let $X$ be a simplicial complex with set of vertices
$V(X)$. For $W \subset V(X)$, the {\em induced subcomplex $X[W]$} is the subcomplex consisting of all faces of $X$ whose vertices
are contained in $W$.

A $d$-dimensional pure simplicial complex is said to be a {\em weak pseudomanifold} if each of its $(d - 1)$-faces is in at most two facets. For a $d$-dimensional weak
pseudomanifold $X$, the {\em boundary} $\partial X$ of $X$ is the pure subcomplex of $X$ whose facets are those $(d-1)$-dimensional faces of $X$ which are contained in a single facet of $X$. The {\em dual graph} $\Gamma(X)$ of a weak pseudomanifold $X$ is the graph whose vertices are the facets of $X$, and where the vertices
corresponding to any two facets of $X$ are adjacent in $\Gamma(X)$ if these two
facets intersect in a face of codimension one. A {\em pseudomanifold} is a weak
pseudomanifold with a connected dual graph. All connected triangulated manifolds
are automatically pseudomanifolds.

If $X$ is a $d$-dimensional simplicial complex then, for $0\leq i
\leq d$, the {\em $i$-skeleton} of $X$ is defined as
${\rm skel}_{i}(X) := \{\alpha\in X \,: \, \dim(\alpha) = i\}$.
The number of $i$-dimensional faces of $X$ is denoted by $f_i = f_i(X)$.
The vector $f(X) := (f_0, \dots, f_d)$ is called the {\em face
vector} of $X$ and the number $\chi(X) := \sum_{i=0}^{d} (-1)^i
f_i$ is called the {\em Euler characteristic} of $X$. As is well
known, $\chi(X)$ is a topological invariant, i.e., it depends only
on the homeomorphism type of the geometric carrier $|X|$ of $X$.

A simplicial complex $X$ is
said to be {\em $k$-neighbourly} if any $k$ vertices of $X$ form a
$(k-1)$-face of $X$. A 2-neighbourly simplicial complex is also called a
{\em neighbourly} simplicial complex.

A $d$-dimensional simplicial complex $X$ is called {\em minimal}
if $f_0(X) \leq f_0(Y)$ for every triangulation $Y$ of the
geometric carrier $|X|$ of $X$. We say that $X$ is {\em strongly
minimal} if $f_i(X) \leq f_i(Y)$, $0\leq i \leq d$, for all such
$Y$.

\subsection{Stacked and locally stacked triangulated manifolds}

A triangulated $(d+1)$-manifold $N$ with non-empty boundary is said to be
{\em stacked}, if all its interior faces have dimension $\geq d$, i.e.,
${\rm skel}_{d-1}(N) = {\rm skel}_{d-1}(\partial N)$.
A closed triangulated
$d$-manifold $M$ is said to be \textit{stacked} if $M = \partial N$
for some stacked $(d+1)$-manifold $N$.

By a \textit{stacked sphere} we mean a stacked
triangulated sphere. A triangulated manifold is said to be \textit{locally
stacked} if each vertex link is a stacked sphere. The class $\mathcal{K}(d)$
is the set of all locally stacked closed triangulated $d$-manifolds \cite{wa}.

A triangulated ball $B$ is stacked if its dual graph
is a tree. (This can be shown by induction on the number of facets of $B$.)
We make heavy use of this characterisation of stacked-balls
throughout this article.

Let $\overline{{\mathcal K}}(d+1)$ be the class of all simplicial complexes
whose vertex links are stacked $d$-balls.
It is straightforward to show that a stacked $(d+1)$-manifold with boundary is
a member of $\overline{{\cal K}}(d+1)$. Conversely, all
members $N \in \overline{{\mathcal K}}(d+1)$ are stacked. To see this note
that if $N \in \overline{{\cal K}}(d+1)$ then $N$ is a triangulated
$(d+1)$-manifold with boundary, all whose vertices are on the boundary
$\partial N$, and with ${\rm skel}_{d-2}({\rm lk}_N(v)) = {\rm skel}_{d - 2}(\partial ({\rm lk}_N(v)))$ for each vertex $v$ of $N$. This implies ${\rm skel}_{d-1}(N) = {\rm skel}_{d-1}(\partial N)$. Thus, $\overline{{\cal K}}(d+1)$ is the set of stacked $(d+1)$-manifolds with boundary.

Since every stacked $(d+1)$-manifold with boundary is in
$\overline{{\cal K}}(d+1)$, it follows that every closed stacked $d$-manifold
$M$ is a member of $\mathcal{K}(d)$ for $d\geq 2$. In \cite{ka}, Kalai
showed that for $d\geq 4$, all members of $\mathcal{K}(d)$ are stacked.
In contrast, there are locally stacked closed 3-manifolds (such as Example 6.2
in \cite{BDS}) which are not stacked. Recently, the second author and Murai showed
that a closed $d$-manifold is stacked if and only if it is in Walkup's class
${\mathcal H}^{d}$ of manifolds defined in \cite{wa} (see \cite[Corollary 4.9]{dm}).
The following statement follows from that result.

\begin{prop}[Datta-Murai \cite{dm}] \label{prop:dm}
For $d\geq 2$, if $M$ is a closed stacked $d$-manifold then $M$ triangulates
$(S^{d-1}\times S^1)^{\# k}$ or $(\TPSSDS)^{\# k}$ for some $k\geq 0$.
\end{prop}

\subsection{Tight and tight-neighbourly triangulations}
\label{ssec:tntntrigs}

In this section we briefly review the notions of tight, and
tight-neighbourly triangulations,
and discuss how tight and tight-neighbourly triangulations
relate to the concepts of stackedness and local stackedness.
Moreover, we list a number of recent
results that allow us to
characterise tight triangulations and their underlying
topologies in dimension three.

\smallskip

For a field $\FF$, a simplicial complex $X$ is
called {\em tight with respect to} $\FF$ (or {\em $\FF$-tight}) if
(i) $X$ is connected, and (ii) for all induced subcomplexes $Y$ of
$X$ and for all $j\geq 0$, the morphism $H_{j}(Y; \FF) \to
H_{j}(X; \FF)$ induced by the inclusion map $Y \hookrightarrow X$
is injective. If $X$ is $\QQ$-tight then it is $\FF$-tight for all
fields $\FF$ \cite{bd17}. Call $X$ {\em tight}
if it is tight with respect to $\ZZ_2=\mathbb{Z}/2\ZZ$.

For $d\geq 3$, a triangulated $d$-manifold $M$ is called
\textit{tight-neighbourly} if $\binom{\,f_0(M)-d- 1\,}{2} =\binom{d
+2}{2}\beta_1(M; \mathbb{Z}_2)$. Such a triangulation must necessarily be
minimal, but not every manifold can have a tight-neighbourly triangulation.

%
%

The most simple sufficient condition for triangulations to be tight involves
bounded manifolds as well as the condition of being stacked.

\begin{theorem}[Bagchi \cite{BaTight}] \label{prop:bb}
Let $M$ be a neighbourly stacked triangulated $d$-manifold with boundary,
$d\geq 3$. Then $M$ is $\FF$-tight for any field $\FF$.
\end{theorem}

It is worthwhile mentioning that the boundary of any such bounded neighbourly
stacked triangulation is tight as well. Finding sufficient conditions for
tightness of closed triangulations is slightly more intricate. Firstly,
we have the following result.

\begin{theorem}[Effenberger \cite{ef}] \label{prop:ef}
  For $d\neq 3$, the neighbourly members of $\Kd{d}$ are tight.
\end{theorem}

Theorem~\ref{prop:ef} is not true for $d=3$ (see for example
\cite[Example 6.2]{BDS}). However, the following more restrictive
version holds in the three-dimensional case. This is where the notion of
tight-neighbourly triangulations comes into play.

\begin{theorem}[Burton-Datta-Singh-Spreer \cite{bdss}] \label{prop:bdss}
Let $X$ be a tight-neighbourly triangulated $3$-manifold. Then $X$ is tight.
\end{theorem}

Theorem~\ref{prop:bdss} is complemented by the following equivalence
condition for the locally stacked $3$-dimensional case.

\begin{theorem}[Bagchi-Datta \cite{bd16}]\label{prop:bd}
If $X$ is a neighbourly member of $\Kd{3}$, then $X$ is tight if and only if
$X$ is tight-neighbourly.
\end{theorem}

Recently, the following classification theorem was proved in dimension 3.

\begin{theorem}[Bagchi-Datta-Spreer \cite{BDS2016}] \label{prop:bds2}
A triangulated closed $3$-manifold $M$ is $\FF$-tight if and only if $M$ is $\FF$-orientable, neighbourly and stacked.
\end{theorem}

Moreover, we have the following connection between stacked and tight-neighbourly
triangulations due to recent work by Murai.

\begin{theorem}[Murai \cite{Mu}] \label{prop:Mu}
Let $X$ be a closed triangulated manifold of dimension $d\geq 3$. Then $X$ is
tight-neighbourly if and only if $X$ is stacked and neighbourly.
\end{theorem}

In summary, assuming the necessary condition of being neighbourly, the
properties of stackedness, local stackedness, and tight-neighbourliness
of a triangulated manifold are closely interconnected with the property of
being tight.

\smallskip

For a triangulated $3$-manifold $M$, tightness is thus determined by its
first Betti number. More explicitly, $M$ is tight if and only if
\begin{align*}
10\beta_1(M; \mathbb{Z}_2) & = \binom{\,f_0(M)-4\,}{2}.
\end{align*}

In particular, the first Betti number can be computed in polynomial time, and
thus tightness of $3$-manifolds is polynomial time decidable (see
\cite{BBDSS2016} for an alternative polynomial time algorithm independent of
Theorem~\ref{prop:bds2}). Furthermore, we now know
that any tight triangulation of a $3$-manifold $M$ must be homeomorphic to
$(\TPSSDS)^{\# k}$, or $(S^{d-1}\times S^1)^{\# k}$, where
$k = \binom{\,f_0(M)-4\,}{2}/10$.



\subsection{The complex \boldmath{${\cal K}(G,{\cal T})$}} \label{subsec:K(G,T)}

In this section we describe a representation of a weak
pseudomanifold ${\cal K}$ in terms of its dual graph $G$ and its (stacked)
vertex links, given by a collection of trees ${\cal T}$. This representation
was first used by the second and third authors \cite{bdns}. The complex
${\cal K}(G,{\cal T})$, defined below, is the central object of our construction
principle for tight manifolds presented in this article.

Let $G$ be a graph and ${\cal T}=\{T_i\}_{i\in {\cal I}}$ be a family of
induced subtrees of $G$, such that every vertex of $G$ is contained in
exactly $d+2$ trees of ${\cal T}$ and any two adjacent vertices appear together
in exactly $d+1$ trees of ${\cal T}$. We consider the $(d+1)$-dimensional
simplicial complex ${\cal K}(G,{\cal T})$ given by:
\begin{equation}\label{eq:scdual}
{\cal K}(G,{\cal T}) := \{\{i: u\in T_i\}: u\in V(G)\}.
\end{equation}
We denote the facet $\{i: u\in T_i\}$ by $\hat{u}$ for $u\in V(G)$.
Our construction is based on the following result from \cite{bdns}.

\begin{theorem}\label{thm:construction}
Let $G$ be a graph and ${\cal T} = \{T_i\}_{i=1}^{n}$ be a  family
of $(n-d-1)$-vertex induced subtrees of $G$, any two of which
intersect. Suppose that
\begin{enumerate}[(i)]
  \item each vertex of $G$ is in exactly $d+2$ trees of $\cal T$; and
  \item for any two vertices $u\neq v$ of $G$, $u$ and $v$ are jointly contained
    in $d+1$ members of $\cal T$ if and only if $\langle u,v \rangle$ is an edge
    of $G$.
\end{enumerate}
Then ${\cal K}(G,{\cal T})$ is a neighbourly member of $\overline{\cal K}(d+1)$,
with dual graph $\Gamma({\cal K}(G,{\cal T}))$ isomorphic to $G$.
\end{theorem}

\section{Constructing tight triangulations of handlebodies}
\label{sec:constrPrinciple}

In this section we present a method to construct examples of tight
triangulated $d$-manifolds obtained as boundaries of tight $(d+1)$-dimensional
handlebodies with $n=(d+1)((d+2)k+2)+1$, for $k\geq 0$, vertices and first Betti number
 \begin{equation} \label{eq:Betti}
    \beta_1 = \begin{cases}
	{\binom{(d+1)((d+2)k+1)+1}{2}}/{\binom{d+2}{2}} & \text{ if } d\geq 3, \\
	2\binom{(d+1)((d+2)k+1)+1}{2}/\binom{d+2}{2} & \text{ if } d = 2,
      \end{cases}
  \end{equation}
and with $\mb{Z}_n = \langle (1, 2, \ldots , n) \rangle$ as the group of
automorphisms acting transitively on their sets of vertices
$\{ 1, 2, \ldots, n \}$. The construction principle is defined in
Section~\ref{ssec:proof}.
The main ingredient for the proof of correctness
is a modification of Theorem~\ref{thm:construction}, which we give in
Lemma~\ref{lem:construction2}. The method is feasible for small values of $k$
and leads to several new examples of tight triangulations
presented in Section~\ref{sec:results}.

However, before we can describe the construction principle and the new tight
triangulations we first need to introduce some extra notation.

\subsection{A family of candidates for the dual graph} \label{subsec:dualgr}

Let $k\geq 0$, $n=(d+1)((d+2)k+2)+1$, and let $m_0,m_1,\ldots, m_k$ be
invertible elements in $\mathbb{Z}_n$. Let $(r,\theta)$ denote the point in the
two dimensional plane with polar coordinates $(1+r,\theta)$. Define
$\theta_n=\exp(2\pi/n)$ and let $G(d,k;m_0,\ldots,m_k)$ be the
$n((d+1)k+1)$-vertex ``spiderweb-graph''. This graph consists of sets of
$n$ vertices each lying on $(d+1)k+1$ concentric circles with coordinates
$(i,\ell\theta_n)$, $\ell \in \mathbb{Z}_n$, $1 \leq i \leq (d+1)k+1$. We denote
these vertices by $v_i (\ell)$. As for the edges of $G(d,k;m_0,\ldots,m_k)$ we
have

\begin{enumerate}[(i)]
\item $k+1$ ``orbit'' cycles
$$
C_i := v_{(d+1)i}(0)\mbox{-} v_{(d+1)i} (m_i)\mbox{-} v_{(d+1)i}(2 m_i)\mbox{-}
\cdots \mbox{-} v_{(d+1)i}((n-1)m_i)\mbox{-} v_{(d+1)i}(0), ~~ 0\leq i\leq k,
$$
and

\item $n$ ``radial'' paths
$$
P_j := v_0(j)\mbox{-} v_1(j)\mbox{-} \cdots \mbox{-} v_{(d+1)k}(j), ~~ j\in
\mathbb{Z}_n.
$$
\end{enumerate}

See Figure~\ref{fig:g49} for a picture of the graph with $d=3$, $k=2$, $n=49$,
$m_0=10$, $m_1=6$, and $m_2=1$.

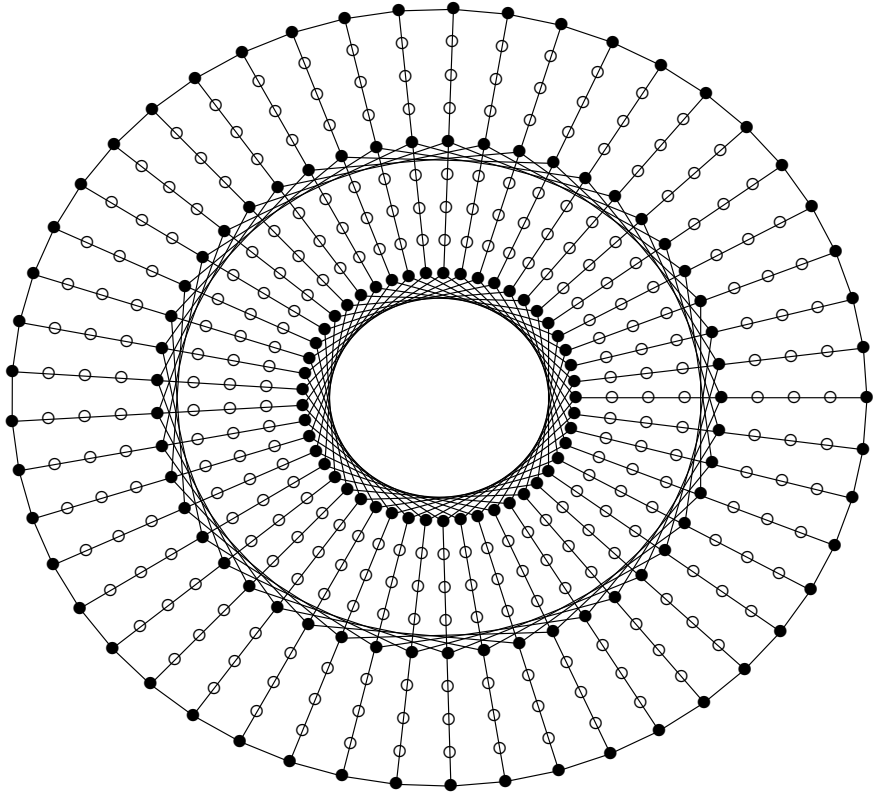
\begin{figure}[ht!]
\centering
\scalebox{1.1}{%
\begin{tikzpicture}[xscale=1.1]
\pgfmathsetmacro{\rA}{4.7}
\pgfmathsetmacro{\rB}{3.1}
\pgfmathsetmacro{\rC}{1.5}
\pgfmathsetmacro{\aA}{7.34}
\pgfmathsetmacro{\aB}{44.081}
\pgfmathsetmacro{\aC}{73.469}

\foreach \i in {0,...,48} {
\draw ({(\aA*\i)}:\rA) -- ({(\aA*(\i+1))}:\rA);
\draw ({(\aB*\i)}:\rB) -- ({(\aB*(\i+1))}:\rB);
\draw ({(\aC*\i)}:\rC) -- ({(\aC*(\i+1))}:\rC);

\draw ({(\aA*\i)}:\rA) -- ({(\aA*\i)}:\rB);
\draw ({(\aA*\i)}:\rB) -- ({(\aA*\i)}:\rC);

\draw ({(\aA*\i)}:\rA) node {$\bullet$};
\draw ({(\aB*\i)}:\rB) node {$\bullet$};
\draw ({(\aC*\i)}:\rC) node {$\bullet$};

\foreach \j in {0.25, 0.50, 0.75} {
\draw ($({(\aA*\i)}:\rA)!\j!({(\aA*\i)}:\rB)$) node {$\circ$};
\draw ($({(\aA*\i)}:\rB)!\j!({(\aA*\i)}:\rC)$) node {$\circ$};
}

}
\end{tikzpicture}
}
\caption{The graph $G(3,2;10,6,1)$}
\label{fig:g49}
\end{figure}

Observe that the automorphism group of the graph $G(d,k;m_0,\ldots,m_k)$
contains $n$ rotations generated by the following permutation (written in
cycle notation):
\begin{equation}\label{eq:auto}
  \varphi := \prod_{i=0}^{(d+1)k} \left ( v_i (0), \, v_i (1), \, \ldots , \,
  v_i (n-1) \right ).
\end{equation}

\subsection{Families of induced subtrees of \boldmath{$G(d,k; m_0,\ldots,m_k)$}}
\label{subsec:subtrees}

To construct neighbourly members of $\overline{\cal K}(d+1)$, we define
families of induced subgraphs of $G(d,k; m_0,\ldots,m_k)$ of a particular type.

For $0\leq i\leq k-1$ and $1\leq \ell\leq d$, we say that the path
$$
v_{(d+1)i} (j)\mbox{-}  v_{(d+1)i+1} (j)\mbox{-} \cdots \mbox{-}
v_{(d+1)i+\ell} (j)
$$
is the {\em outward} path of length $\ell$ at vertex $v_{(d+1)i} (j)$.
Similarly, for $1\leq i\leq k$ and $1\leq \ell\leq d$, we say that the path
$$
v_{(d+1)i} (j) \mbox{-}  v_{(d+1)i-1} (j)\mbox{-}  \cdots \mbox{-}
v_{(d+1)i-\ell} (j)
$$
is the {\em inward} path of length $\ell$ at vertex $v_{(d+1)i} (j)$.

A collection ${\cal D}=\{(\sigma_i,\tau_i):1\leq i\leq k\}$, where
$\sigma_i$ and $\tau_i$ are permutations of the set $\{0,1, \ldots , d\}$, is called
a $k$-{\em deck} of permutations.

\begin{defn}\label{defn:tree}
For a $k$-deck ${\cal D} =
\{(\sigma_i,\tau_i): 1\leq i\leq k\}$, let $\tr{j}$ be the induced subgraph of
$G(d,k; m_0,\ldots,m_k)$ spanned by the following paths (see
Figure~\ref{fig:tree} for a picture):

\begin{enumerate}[(i)]
\item {\em orbit paths} $v_{(d+1)i}(j)\mbox{-}  v_{(d+1)i}(j+m_i) \mbox{-}
  \cdots \mbox{-} v_{(d+1)i}(j+(d+1)m_i)$, $0\leq i\leq k$;

\item a {\em radial path} $P_j = v_{0} (j) \mbox{-} v_{1}(j) \mbox{-} \cdots
  \mbox{-}  v_{(d+1)k} (j)$;

\item {\em outward paths} of length $\tau_{i+1} (\ell-1)$ at vertex
      $v_{(d+1)i}(j+\ell \, m_i)$ for  $0\leq i\leq k-1$ and $1\leq
      \ell\leq d+1$;

\item {\em inward paths} of length $\sigma_{i}(\ell-1)$ at
      vertex $v_{(d+1)i} (j+\ell \, m_i)$ for $1\leq i\leq k$ and
      $1\leq \ell\leq d+1$.
  \end{enumerate}
\end{defn}

Observe that any subgraph $\tr{j}$ can be obtained by applying the $j$-th
power of the automorphism $\varphi$ to the subgraph $\tr{0}$, that is,
$\tr{j} = \varphi^j(\tr{0})$.

\smallskip

Given a family $\tr{j}$, $0 \leq j < n$, of induced subgraphs of
$G(d,k;m_0,\ldots,m_k)$, we are interested in how its members intersect.
For this reason we define what we call the {\em span} of a
$k$-deck of permutations ${\cal D}$, denoted by $\operatorname{sp}({\cal D})$,
as:
\begin{equation}\label{eq:span}
  \operatorname{sp}({\cal D}) := \displaystyle \bigcup_{i=1}^k
  \operatorname{sp}(\sigma_i,\tau_i) \cup \{\pm \ell \cdot m_j \in
  \mb{Z}_n: 0\leq \ell
  \leq (d+1), \,\, 0\leq j \leq k\},
\end{equation}
where $\operatorname{sp}(\sigma_i,\tau_i)$ is defined to be the subset of
$\mb{Z}_n$ given by
$$\operatorname{sp}(\sigma_i, \tau_i) := \{\pm ((q+1)m_{i-1}-(p+1)m_i) :
0\leq p,q\leq d, \,\, \sigma_i(p)+\tau_i(q)\geq d+1 \}.$$

For a motivation of the definition of $\operatorname{sp}({\cal D}) \subseteq 
\mathbb{Z}_n$ see Lemma~\ref{lem:span} where we prove that for any
$j \in \operatorname{sp}({\cal D})$, the subgraphs $\tr{j}$ and $\tr{0}$ intersect.

\begin{figure}[ht!]
\centering
\begin{tikzpicture}[xscale=1.2]

\foreach \i in {0,1,2,3,4,5,6,7,8} {
\coordinate (VV_\i) at (0,0.75*\i);
\draw (VV_\i) node {\tiny $\bullet$};
}

\foreach \i in {1,2,3,4} {
\coordinate (VH1_\i) at (1.5*\i,6);
\draw (VH1_\i) node {\tiny $\bullet$};
}

\foreach \i/\j in {1/20,2/40,3/11,4/31} {
\coordinate (VH2_\i) at (1.5*1.2*\i,3);
\draw (VH2_\i) node {\tiny $\bullet$};
}

\foreach \i/\j in {1/41,2/33,3/25,4/17} {
\coordinate (VH3_\i) at (1.5*\i,0);
\draw (VH3_\i) node {\tiny $\bullet$};
}

\coordinate (V11) at (1*1.5,7*0.75);
\coordinate (V13) at (3*1.5,7*0.75);
\coordinate (V23) at (3*1.5,6*0.75);
\coordinate (V14) at (4*1.5,7*0.75);
\coordinate (V24) at (4*1.5,6*0.75);
\coordinate (V34) at (4*1.5,5*0.75);
\coordinate (V240) at (2.4*1.5, 6*0.75);
\coordinate (V340) at (2.4*1.5, 5*0.75);
\coordinate (V311) at (3.6*1.5, 5*0.75);
\coordinate (V131) at (4.8*1.5, 7*0.75);
\coordinate (V231) at (4.8*1.5, 6*0.75);
\coordinate (V331) at (4.8*1.5, 5*0.75);
\coordinate (V520) at (1.2*1.5, 3*0.75);
\coordinate (V540) at (2.4*1.5, 3*0.75);
\coordinate (V640) at (2.4*1.5, 2*0.75);
\coordinate (V531) at (4.8*1.5, 3*0.75);
\coordinate (V631) at (4.8*1.5, 2*0.75);
\coordinate (V731) at (4.8*1.5, 1*0.75);
\coordinate (V741) at (1*1.5, 1*0.75);
\coordinate (V725) at (3*1.5, 1*0.75);
\coordinate (V625) at (3*1.5, 2*0.75);
\coordinate (V525) at (3*1.5, 3*0.75);
\coordinate (V717) at (4*1.5, 1*0.75);
\coordinate (V617) at (4*1.5, 2*0.75);

\draw (VV_0) -- (VV_8)
    (VV_8) -- (VH1_4)
    (VV_4) -- (VH2_4)
    (VV_0) -- (VH3_4);

\draw (VH1_1) -- (V11)
    (VH1_3) -- (V23)
    (VH1_4) -- (V34)
    (VH2_2) -- (V240)
    (VH2_3) -- (V311)
    (VH2_4) -- (V131);

\draw (VH2_1) -- (V520)
    (VH2_2) -- (V640)
    (VH2_4) -- (V731)
    (VH3_1) -- (V741)
    (VH3_3) -- (V525)
    (VH3_4) -- (V617);

\foreach \i in {V11,V13,V23,V14,V24,V34} {
    \draw (\i) node {\tiny $\bullet$};
}

\foreach \i in {V240,V340,V311,V131,V231,V331} {
    \draw (\i) node {\tiny $\bullet$};
}

\foreach \i in {V520,V540,V640,V531,V631,V731} {
    \draw (\i) node {\tiny $\bullet$};
}

\foreach \i in {V741,V725,V625,V525,V717,V617} {
    \draw (\i) node {\tiny $\bullet$};
}

\foreach \i in {0,1,2,3,4,5,6,7,8} {
    \draw (VV_\i) node[left] {\tiny $v_{\i} (0)$};
}

\foreach \i in {1,2,3,4} {
    \draw (VH1_\i) node[above] {\tiny $v_{8} (\i)$};
}

\foreach \i/\j in {1/20,2/40,3/11,4/31} {
    \draw (VH2_\i) node[below left] {\tiny $v_{4} (\j)$};
}

\foreach \i/\j in {1/41,2/33,3/25,4/17} {
    \draw (VH3_\i) node[below] {\tiny $v_{0} (\j)$};
}
\end{tikzpicture}
\caption{\small Illustration of the tree ${\cal D}T_0$ in $G(3,2;41,20,1)$.
The corresponding $2$-deck of permutations of $\{0,1,2,3\}$ is
${\cal D}=\{((1,2,0,3), (1,0,3,2)), ((1,0,2,3), (0,2,1,3))\}$.
}
\label{fig:tree}
\end{figure}
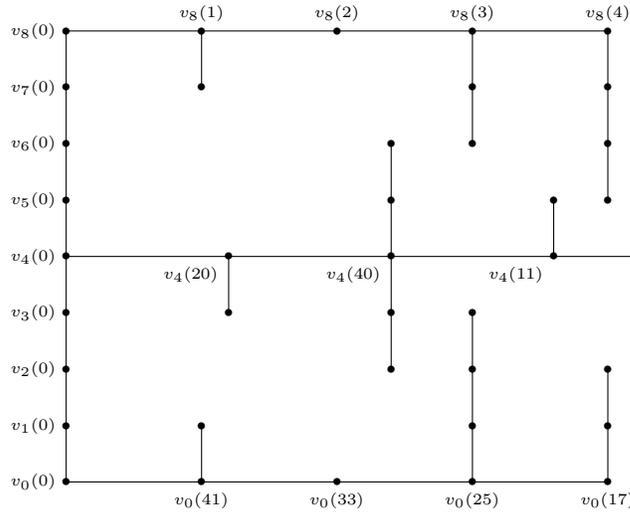

\begin{lemma}
  \label{lem:tree}
  If $\operatorname{sp}({\cal D})=\mathbb{Z}_n$, then the subgraph ${\cal D}T_j$
  is a tree for all $j=0,\ldots,n-1$.
\end{lemma}

\begin{proof}
  In a first step we show that $\operatorname{sp}({\cal D})=\mathbb{Z}_n$
  implies that all $k+1$ sets in Equation~(\ref{eq:span}) are pairwise disjoint. Note that
  $|\operatorname{sp}(\sigma_i,\tau_i)|\leq 2 \binom{d+1}{2} = d(d+1)$. So, the
  number of elements in $\operatorname{sp}({\cal D})$ is at most
  $$ d(d+1)k+ 2(d+1)(k+1) + 1=(d+1)((d+2)k+2) + 1=n.$$
  Thus all the sets in the union must be mutually disjoint
  (over $\mathbb{Z}_n$) to achieve $\operatorname{sp}({\cal D})=\mathbb{Z}_n$.

  \medskip
  Now we proceed to show that ${\cal
  D}T_j$ is a tree for all $j=0,\ldots,n-1$. Note that steps (i) and
  (ii) of the construction in Definition \ref{defn:tree} yield a tree
  consisting of the path $P_j$ and arcs of length $d+1$ of cycles $C_i$ for
  $0\leq i\leq k$. In steps (iii) and (iv) we attach outward and inward
  paths to the arcs attached in step (i). Note that the resulting graph is
  connected. For a cycle to occur, two paths added in steps (iii) and (iv)
  must be subpaths of the same radial path $P_{\ell}$ for some $\ell$. For two outward paths,
  or two inward paths to be subpaths of the same path $P_{\ell}$, we must have
$$
  j+(t+1)m_i=j+(t'+1)m_i
$$
  for some $0\leq t<t'\leq d$, $0\leq i\leq k$. As all values $m_i$ are
  invertible in $\mathbb{Z}_n$, this is not possible. For an outward path and a
  inward path to be subpaths of the same  path $P_{\ell}$, we must have
$$
j+(t+1)m_{i-1} = j+(t'+1)m_i
$$
  for some $0\leq t,t'\leq d$, $0\leq i\leq k$. But then the sets
  $\{\pm \, \ell \, m_i: 1\leq \ell \leq d+1\}$ and $\{\pm \, \ell \, m_{i-1}:
  1\leq \ell \leq (d+1)\}$ are not disjoint, a contradiction. Thus the resulting graph   is connected, without a cycle -- and hence a tree.
\end{proof}

\begin{lemma}
  \label{lem:span}
  For all $j\in \operatorname{sp}({\cal D})$, the trees $\tr{0}$ and
  $\tr{j}$ intersect.
\end{lemma}

\begin{proof}
  If $j= \pm \, \ell \, m_i$ for some $0\leq \ell \leq d+1$,
  $1 \leq i \leq k$, then either vertex $v_{(d+1)i} (\ell\,m_i)$ or vertex
  $v_{(d+1)i} (0)$ must be common to both $\tr{0}$ and $\tr{j}$. Hence they
  intersect.

  Let $j = (q+1) m_{i-1} - (p+1) m_i$, $0\leq p,q\leq d$,
  $\sigma_i(p)+\tau_i(q)\geq d+1$. Then $\tr{j}$
  contains an inward path of length $\sigma_i(p)$ at $v_{(d+1)i} (j+(p+1)m_i)$,
  and $\tr{0}$ contains an outward path of
  length $\tau_i(q)$ at $v_{(d+1)(i-1)} ( (q+1)m_{i-1} )$.
  Since $j = (q+1) m_{i-1} - (p+1) m_i$ we have
  $$ (q+1)m_{i-1} = j+(p+1)m_i, $$
  and since $\sigma_i(p)+\tau_i(q)\geq d+1$, the inward and outward paths
  intersect. Similarly, it can be shown that $\tr{0}$ and $\tr{j}$ intersect
  whenever $j=-(q+1)m_{i-1}+(p+1)m_i$, $\sigma_i(p)+\tau_i(q)\geq d+1$.
\end{proof}

\subsection{The construction principle} \label{ssec:proof}

\begin{theorem}   \label{thm:perselect}
  Let $G = G(d,k;m_0,\ldots,m_k)$ and let 
  ${\cal D}=\{(\sigma_i,\tau_i): 1\leq i\leq k\}$ be a $k$-deck of
  permutations such that:
  \begin{enumerate}[(i)]
    \item $\operatorname{sp}({\cal D}) = \mb{Z}_n$;
    \item $\tau_{i}(t)+\sigma_{i-1}(t)\geq 1$ for $2\leq i\leq k,
      0\leq t\leq d-1$;
    \item $\sigma_{i}(d),\tau_{i}(d) \geq 1$ for $1\leq i\leq k$.
  \end{enumerate}
  Then ${\cal K}(G,{\cal D}T)$ is a neighbourly member of $\overline{\cal
  K}(d+1)$ where ${\cal D}T = \{\tr{j}: 0\leq j\leq n-1\}$.
\end{theorem}

To prove Theorem~\ref{thm:perselect}, we first present an equivalent of
Theorem~\ref{thm:construction}.

\begin{lemma}  \label{lem:construction2}
  Let $G$ be a graph and ${\cal T}=\{T_i\}_{i=1}^n$ be a family of
  $(n-d-1)$-vertex induced subtrees of $G$, any two of which intersect.
  Suppose that
  \begin{enumerate}[(i)]
    \item each vertex of $G$ is in exactly $d+2$ members of ${\cal T}$,
    \item any two adjacent vertices $u$ and $v$ occur together in exactly
      $d+1$ members of ${\cal T}$, and
    \item for a vertex $u\in V(T), T\in {\cal T}$, we have
      $\deg_G(u)-\deg_T(u)\leq 1$.
  \end{enumerate}
  Then ${\cal K}(G, {\cal T})$ is a neighbourly member of
  $\overline{\cal K}(d+1)$ with $\Gamma({\cal K}(G, {\cal T}))\cong G$.
\end{lemma}

\begin{proof}
  Let $T_i\in {\cal T}$ be a tree. For a vertex $r\in T_i$, define the oriented
  tree $T_i(r)$ with directed edges $\ovr{\langle u,v \rangle}$ where
  $\langle u,v \rangle\in T_i$ and $v$ is closer to $r$ than $u$. Define the
  {\em label} $l(\ovr{\langle u,v \rangle})$ to be the unique
  element of $\hat{u}\backslash \hat{v}$ (it follows from conditions (i) and
  (ii) that $\hat{u}\backslash \hat{v}$ must have exactly one element). We
  prove that all edges in $T_i(r)$ have distinct labels.

  There are $d+1$ other trees that intersect $T_i$ in $r$. Let $T_j\in {\cal T}$
  be a tree that does {\em not} intersect $T_i$ in $r$. Since any two trees in
  ${\cal T}$ intersect, there is a vertex $w\neq r$ which is common to $T_i$
  and $T_j$. Since $w, r \in T_i$ there is a unique path from $w$ to $r$ in
  $T_i$. Furthermore, since $r \not \in T_j$ but $w \in T_j$ one of the edges
  in this path in $T_i(r)$ from $w$ to $r$ must have the label $j$.
  Since there are $n-d-2$ such trees and also $n-d-2$ edges in $T_i(r)$, we
  conclude that all labels must be distinct. Furthermore the labels are
  different from the ones seen at $r$. This property of oriented trees was
first proved in \cite{si}, which we have slightly adapted to suit the
current setting.

  \smallskip

  We now prove that $(G,{\cal T})$ satisfies the conditions in Theorem
  \ref{thm:construction}. Essentially, we need to show that $|\hat{u}\cap
  \hat{v}|=d+1$ implies that $\langle u,v \rangle$ is an edge in $G$.
  Suppose $u$ and $v$ are vertices in $G$
  such that $|\hat{u}\cap \hat{v}|={d+1}$, i.e., $\hat{u} = \{ a_1, \ldots ,
  a_{d+1}, r \}$ and $\hat{u} = \{ a_1, \ldots , a_{d+1}, s \}$ for some
  $r,s \in \mathbb{Z}_n$. Furthermore, assume that $\langle u,v \rangle$ is
  not an edge of $G$. Let $T_i$ be one of the trees containing both $u$ and
  $v$, i.e., $i \in \{a_1, \ldots , a_{d+1} \}$, and let $w$ be an internal
  vertex of the path from $u$ to $v$ in $T_i$.

  \medskip

\noindent {\em Claim: } $\hat{u}\cap \hat{v}\subseteq \hat{u}\cap \hat{w}$.

  If possible, let $j\in (\hat{u}\cap \hat{v})\backslash (\hat{u}\cap
  \hat{w})$. Then $j\in \hat{u},\hat{v}$ but $j\not\in \hat{w}$. Hence, in
  the oriented tree $T_i(w)$, there exist edges on the paths from $u$ to $w$
  and $v$ to $w$ with label $j$. But this contradicts the
  uniqueness of labels on the edges of $T_i(w)$. This proves the claim.

  \medskip

  Let $u,w,z$ be the first three vertices on the path from $u$ to $v$ in
  $T_i$ (note that, possibly, $z = v$). Hence, by the above claim
  we have $\hat{w} = \{ a_1, \ldots ,
  a_{d+1}, j \}$ and $\hat{z}= \{ a_1, \ldots , a_{d+1}, k \}$ for some
  $j,k \in \mathbb{Z}_n$. Then $l (\ovr{\langle u,w \rangle}) = r$ and
  $l (\ovr{\langle w,z \rangle}) = j$.
  We show that $\{r,j\}\in \hat{u}\backslash \hat{v}$, a contradiction to
  $|\hat{u}\cap \hat{v}|=d+1$.

  By the above claim it suffices to show that
  $\{r,j\}\in \hat{u}\backslash \hat{z}$.
  First note that we have $j \in \hat{w}$ but $j \not\in \hat{z}$, and
  $r \in \hat{u}$, but $r \not\in \hat{w}$ by definition. Hence,
  $r \not \in \hat{u} \cap \hat{w}$ and by the above claim $r\not\in \hat{z}$.
  Now suppose that $j\not\in \hat{u}$. Thus $j\in \hat{w}$ but $j \not\in
  \hat{u},\hat{z}$. It follows that $\deg_G(w)\geq \deg_{T_k}(w)+2$,
  a contradiction to assumption (iii). Thus $\{r,j\}\subseteq
  \hat{u}\backslash \hat{w}$ and $|\hat{u}\cap \hat{v}|\leq d$, a
  contradiction to the assumption that $|\hat{u}\cap \hat{v}|=d+1$. This proves
  the lemma.
\end{proof}

We are now in a position to prove Theorem~\ref{thm:perselect}.

\begin{proof}[Proof of Theorem~\ref{thm:perselect}]
  Throughout the proof we make use of the notations
  $G = G(d,k; m_0$, $\ldots,m_k)$ and ${\cal D}T = \{\tr{j}: 0\leq j\leq n-1\}$.
  We show that $(G,{\cal D}T)$ satisfies the conditions in Lemma
  \ref{lem:construction2}.

  \medskip

  \noindent{\em Each tree has $n-d-1$ vertices}:
  From Definition \ref{defn:tree}, the number of vertices in a tree is:

  \begin{equation*}
    \begin{array}{llll}
       & (d+1)(k+1) &+ & {\rm (i)} \\
       & ((d+1)k+1) &+ & {\rm (ii)} \\
       & k(0+1+\cdots+d) &+& {\rm (iii)} \\
       & k(0+1+\cdots+d) && {\rm (iv)} \\
      =& (d+2)(d+1)k + (d+2) && \\
      =& n-d-1. &&
    \end{array}
  \end{equation*}

  \smallskip

  \noindent{\em Any two trees intersect}: From Lemma \ref{lem:span}, it
  follows that any two trees in ${\cal D}T$ intersect.

  \medskip

  \noindent{\em Each vertex appears in $(d+2)$ trees, each edge in
  $(d+1)$-trees}:
  We calculate the number of trees that cover a particular
  vertex $v\in V(G)$. Since $\tr{j}=\varphi^j(\tr{0})$, we see that $|\{j: v\in
  V(\tr{j})\}|=|\{j: \varphi^j(v)\in V(\tr{0})\}|$, which is the same as the number of
  vertices in $\tr{0}$ from the $\varphi$-orbit of $v$. Without loss of generality,
  assume $v=v_{i} (0)$ for some $0\leq i\leq (d+1)k$. When
  $v=v_{(d+1)i} (0)$ for some $0\leq i\leq k$, $T_0$ contains $d+2$ vertices
  from the $\varphi$-orbit
  of $v$. Now let $v=v_{(d+1)i+\ell} (0)$ where $1\leq \ell\leq d$. Let $\varphi_v$
  denote the $\varphi$-orbit of $v$. Note that $\tr{0}$
  intersects $\varphi_v$ at $v$. Other intersections between $\tr{0}$ and
  $\varphi_v$ occur along the inward and outward paths in $\tr{0}$ from
  vertices in $C_{i+1}$ and $C_{i}$ respectively. Now outward paths of
  length $\tau_{i+1}(q)$ at a vertex in $C_i$ intersect $\varphi_v$ if
  $\tau_{i+1}(q)\geq \ell$. Similarly inward paths
  of length $\sigma_{i+1}(p)$ at a vertex in $C_{i}$ intersect
  $\varphi_v$ if $\sigma_{i+1}(p)\geq d+1-\ell$. Recalling that
  $\sigma_{i+1},\tau_{i+1}$ are permutations of $\{0,1,\ldots,d\}$, and that the
  respective paths are distinct, we get that $\tr{0}$ contains
  $1+(d+1)-((d+1)-\ell)+(d+1)-\ell=d+2$ vertices from $\varphi_v$. Similarly, it can
  be shown that $\tr{0}$ contains $d+1$ edges from each edge orbit under
  $\varphi$. As in the vertex case, this
  implies that each edge is covered by exactly $d+1$ trees in ${\cal D}T$.

  \medskip

  \noindent{\em For a vertex $v\in V(\tr{i})$,
  $\deg_G(v)-\deg_{\tr{i}}(v)\leq 1$}:
  Via the automorphism $\varphi$, it is sufficient to prove that for $v\in
  V(\tr{0})$, $\deg_G(v)-\deg_{\tr{0}}(v)\leq 1$. As $\deg_G(v_{i} (j))=2$ for
  $i\not\equiv 0 \mod (d+1)$, there is nothing to prove for those vertices. For
  the remaining vertices we have,
  \begin{equation}
    \deg_G(v_{(d+1)i} (j)) = \begin{cases}
	    3 & \text{if } i\in \{0,k\}, \\
	    4 & \text{ otherwise. }
      \end{cases}
  \end{equation}

  First consider the case when $i\in \{0,k\}$. Note that
  $v_{(d+1)i} (j)\in V(\tr{0})$ for $j \in \{0, m_i, \ldots, (d+1)m_i \}$. The
  vertices $v_{(d+1)i}(0)$ have at least two
  neighbours in $\tr{0}$ (one on the path $P_0$, and the other on the cycle $C_i$).
  Thus the condition holds for these vertices. Similarly the vertices
  $v_{(d+1)i} (m_i), v_{(d+1)i} (2m_i) , \ldots , v_{(d+1)i} (d \, m_i)$ have
  degree at least two in $\tr{0}$ (being internal vertices of a path). The vertex
  $v_{(d+1)i}((d+1)m_i)$ has a neighbour $v_{(d+1)i}(d\,m_i)$
  on $\tr{0}$, and an additional one on the outward (resp., inward) path at
  $v_{(d+1)i}(d\, m_i)$ when $i=0$ (resp., $k$). The preceding claim holds
  because $\sigma_1(d)\geq 1$ and $\tau_k(d)\geq 1$. Now consider the case when
  $i\not\in \{0,k\}$. Then for $0 \leq \ell \leq d-1$ the vertices
  $v_{(d+1)i} ((\ell+1)m_i)$ are internal vertices of a path in $\tr{0}$. It has a further neighbour on an inward or an outward path as
  $\tau_{i}(\ell)+\sigma_{i-1}(\ell)\geq 1$ for $0 \leq \ell < d$. Thus the vertices $v_{(d+1)i} ((\ell+1)m_i)$ for $0 \leq \ell < d$ have degree at least $3$ in $\tr{0}$, and hence they satisfy the condition. The vertices
  $v_{(d+1)i} ((d+1)m_i)$ have a neighbour $v_{(d+1)i}(d\,m_i)$ and one neighbour
  each on outward and inward paths at $v_{(d+1)i}((d+1)m_i)$ as condition
  (c) implies both the paths are non-trivial.

  Thus $(G,{\cal D}T)$ fulfils the requirements of Lemma
  \ref{lem:construction2}, and hence yields a neighbourly $(d+1)$-manifold in
  class $\overline{\cal K}(d+1)$.
\end{proof}

\begin{lemma}\label{lem:automorphism}
  The cyclic group $\langle (0, 1, \ldots , n-1) \rangle$ $(\cong \mathbb{Z}_n)$ acts
  transitively on the manifolds ${\cal K}(G(d,k;m_0,\ldots,m_k), {\cal D}T)$ and
  $\partial {\cal K}(G(d,k;m_0,\ldots,m_k), {\cal D}T)$.
\end{lemma}

\begin{proof}
  Let $\varphi$ be the automorphism of $G:=G(d,k;m_0,\ldots,m_k)$ as in Equation (\ref{eq:auto}).
  We show that $h:i\rightarrow i+1$ (addition modulo $n$) is an automorphism
  of ${\cal K}(G,{\cal D}T)$. Let $\hat{u}=\{i: u\in {\cal D}T_i\}$ be a facet
  of ${\cal K}(G, {\cal D}T)$. Then $h(\hat{u})=\{i+1: u\in
  {\cal D}T_i\}=\{i+1: \varphi(u)\in {\cal D}T_{i+1}\}=\hat{v}$,
  where $v=\varphi(u)$. Since $\varphi$ is an automorphism of $G$, $v$ is a
  facet of ${\cal K}(G, {\cal D}T)$. This proves
  $\langle \, h \, \rangle = \langle (0,1,\ldots,n-1) \rangle \subseteq
  {\rm Aut}({\cal K}(G, {\cal D}T)$. Since $\partial {\cal K}(G, {\cal D}T)$
  contains all the vertices of ${\cal K}(G, {\cal D}T)$,
  $\langle \, h \, \rangle$ is also contained in the automorphism group
  of $\partial {\cal K}(G, {\cal D}T)$.
\end{proof}

\begin{theorem}\label{thm:tight} $\ $
  \begin{itemize}
    \item[(a)] ${\cal K}(G(d, k; m_0, \ldots, m_k), {\cal D}T)$ and
	$\partial{\cal K}(G(d, k; m_0, \ldots, m_k),{\cal D}T)$ are tight triangulations of manifolds, $d\geq 2$.
    \item[(b)] For $d\geq 3$, $\partial{\cal K}(G(d,k;m_0,\ldots, m_k), {\cal D}T)$ is tight-neighbourly and
	triangulates \linebreak  $(S^{d-1}~\times~S^1)^{\# \beta_1}$ or $(\TPSSDS)^{\# \beta_1}$, where $\beta_1$ is the
	first Betti number as in $\eqref{eq:Betti}$.
    \item[(c)] $\partial{\cal K}(G(2,k; m_0,\ldots, m_k), {\cal D}T)$  triangulates $(S^{1}\times S^1)^{\# \beta_1/2}$
	or $(\KBottle)^{\# \beta_1/2}$.
  \end{itemize}
\end{theorem}

\begin{proof}
Let $G = G(d,k; m_0,\ldots, m_k)$.
By construction, all the interior faces of ${\cal K}(G,{\cal D}T)$ have dimensions $d+1$ or $d$. Thus, ${\cal K}(G, {\cal D}T)$ is a neighbourly stacked $(d+1)$-manifold with boundary. So, by Theorem~\ref{prop:bb}, ${\cal K}(G,{\cal D}T)$ is $\FF$-tight for any field $\FF$.

Since ${\cal K}(G, {\cal D}T)$ is a neighbourly stacked $(d+1)$-manifold with boundary and $d\geq 2$, it follows that $\partial{\cal K}(G,{\cal D}T)$ is a neighbourly member of ${\mathcal K}(d)$. Therefore, by Theorems \ref{prop:ef} and \ref{prop:bds2}, $\partial{\cal K}(G,{\cal D}T)$ is tight for all $d\geq 2$. This proves part (a).

Now, assume $d\geq 3$. Since ${\cal K}(G, {\cal D}T)$ is a neighbourly stacked $(d+1)$-manifold with boundary, $\partial{\cal K}(G, {\cal D}T)$ is a neighbourly stacked $d$-manifold. Therefore $\partial{\cal K}(G, {\cal D}T)$ is tight-neighbourly by Theorem \ref{prop:Mu}. Considering that $f_0(\partial{\cal K}(G, {\cal D}T)) = (d+1)((d+2)k+1)$ we thus have $\beta_1(\partial{\cal K}(G, {\cal D}T)) = \binom{(d+1)((d+2)k+1)+1}{2}/\binom{d+2}{2}$. Part (b) now follows from Proposition \ref{prop:dm}.

Finally, assume $d=2$. Since $\partial{\cal K}(G(2,k; m_0,\ldots, m_k), {\cal D}T)$ is neighbourly, the first Betti number $\beta_1$ satisfies \eqref{eq:Betti}. Part (c) now follows from Proposition \ref{prop:dm}.
\end{proof}

\subsection{The algorithm} \label{sec:implement}

We describe an optimized algorithm to search for tight triangulations using
Theorem~\ref{thm:perselect}. Given $k\geq
0$ and $n=(d+1)((d+2)k+2)+1$, our task is the following:
\begin{enumerate}[(i)]
\item Find distinct invertible elements $m_0,\ldots,m_k \in \mathbb{Z}_n$.
Without loss of generality, we may choose $m_k=1$.
\item For each such choice of $m_i$, $0\leq i\leq k$, search for a $k$-deck of
permutations $\{(\sigma_i,\tau_i):1\leq i\leq k\}$ of the set
$\{0,1,\ldots ,d\}$ satisfying the conditions in Theorem~\ref{thm:perselect}.
\end{enumerate}

Note that for $k=0$ this process leads to the unique family of tight triangulations
of sphere bundles due to K\"uhnel~\cite{PSM}. In what follows we can thus assume
that $k>0$.

\medskip
We enumerate the sequences $(m_0,\ldots,m_k)$ in a depth-first manner. To
prune the search tree, we use the following fact, which follows from the
proof of Lemma \ref{lem:tree}.
\begin{equation}\label{eq:prune}
\{\pm \ell \, m_i: 1\leq \ell \leq d+1\}\cap
\{\pm \ell \, m_j: 1\leq \ell \leq d+1\} = \emptyset \text{ for }
i\neq j.
\end{equation}

Having determined a sequence $(m_0,\ldots,m_k)$, we look for a deck of $k$
permutation pairs $\{(\sigma_i,\tau_i):1\leq i\leq k\}$. Again, we enumerate the
permutation pairs in a depth-first manner. The observations below help to
economize the search.

From Theorem~\ref{thm:perselect}, it follows that a valid permutation
$\sigma$ should satisfy $\sigma(\ell)=0$ for one of the values
$0 \leq \ell < d$. Accordingly, we say that $\sigma$ is ``of type $\ell$''. Similarly,
we call a permutation pair $(\sigma,\tau)$ ``of type $(\ell,m)$'' if
$\sigma$ is of type $\ell$ and $\tau$ is of type $m$. Since there are $d!$
permutations of each type, there are $(d!)^2$ permutation pairs of each type. As
there are $d^2$ types of permutation pairs, we have $(d \times d!)^2$
permutations to consider at each level. However, we can
reduce this number.

We call permutation pairs of type
$(\ell,m)$ and $(\ell',m')$ {\em compatible} if $m\neq \ell'$. Observe that
the adjacent permutation pairs $(\sigma_i,\tau_i)$ and
$(\sigma_{i-1},\tau_{i-1})$ in a $k$-deck satisfying
Theorem~\ref{thm:perselect} must be compatible. Thus, apart from the first
level in the depth-first search,
we only have to consider $(d-1)\cdot d\cdot (d!)^2$ compatible permutation pairs.

To enable faster access to compatible permutations at each level, we perform a
pre-processing step of storing them by their type. We store $(d!)^2$
permutations of each type in a contiguous block. Then we stack $d^2$ such
blocks to form a linear array of $(d \times d!)^2$ permutation pairs. The blocks
are stacked following the lexicographic ordering of the type of the permutation
pairs they contain. It can be seen that in this scheme, all the permutation
pairs compatible with a given permutation pair occur as contiguous blocks,
possibly wrapping around at the end of the array.

Finally, we store a permutation pair $(\sigma,\tau)$ as the following set,
which we call its {\em treetype}.
\begin{equation}\label{eq:treetype}
\Delta(\sigma, \tau) = \{(p+1, q+1): \sigma(p)+\tau(q)\geq d+1\}.
\end{equation}
The nomenclature ``treetype'' denotes the fact that $\Delta(\sigma,\tau)$ determines the shape of the tree at a particular level. Given a set of $\binom{d+1}{2}$-tuples $S$, it is purely mechanical to recover permutations $\sigma$ and $\tau$ such that their treetype $\Delta(\sigma,\tau)$ equals $S$. This follows from the observation that if $\sigma(p)=\ell$, then $p+1$ appears as the first coordinate in exactly $\ell$ tuples in $S$. Similarly $\tau(q)=m$ implies that $q+1$ occurs as the second coordinate in exactly $m$ tuples in $S$. For example if $2$ appears as the first coordinate in $3$ tuples of $S$, we conclude $\sigma(2)=3$.


\section{Results} \label{sec:results}

\subsection{\boldmath{$2^{d-1} \cdot \lfloor d/2 \rfloor ! \cdot \lfloor (d-1)/2 \rfloor !$} non-isomorphic tight $d$-manifolds}
\label{ssec:family}

In this section we describe a family of tight triangulations of a closed $d$-manifold with a superexponentially increasing number of members for increasing dimension.

\smallskip

\noindent{\it Notation:} We represent a permutation $\sigma$ of
$\{0,\ldots,d\}$ as the $(d+1)$-tuple $(\sigma(0),\ldots,\sigma(d))$. This is not to be confused with the cyclic permutation as commonly used in algebra.


\begin{lemma} \label{lem:family}
Let $G$ be the graph $G(d,1;d+2,1)$, and let ${\cal D}=\{(\beta,\alpha)\} = \{((0, d, b_2, \ldots, b_{d}),$ $(0, a_1, \ldots, a_{d}))\}$ be a $1$-deck of permutations such that $b_{i+1}+b_{d+1-i}=d$ and $a_i+a_{d+1-i}=d+1$ for all $i\geq 1$. Then ${\cal D}$ satisfies the preconditions of Theorem~\ref{thm:perselect}.
\end{lemma}

The infinite families of tight triangulations presented in \cite{bdns} by the second and third authors are given  by the permutations
 \begin{align} \label{eq:bdns}
  (\beta, \alpha) = ((0,d,\ldots,1), (0,1,\ldots,d)) & \mbox{ and }
 (\beta, \alpha) = ((0,d,\ldots,1), (0,d,\ldots,1)).
   \end{align}
  Observe that both pairs satisfy the conditions of Lemma \ref{lem:family}, and thus both families are contained in this one.

\begin{proof}[Proof of Lemma~\ref{lem:family}]
Since $k=1$, condition {\em (ii)} is vacuously satisfied. Furthermore, condition {\em (iii)} is always satisfied since both $\beta$ and $\alpha$ never have $0$ in the last position. Hence, the only condition that needs to be verified is condition {\em (i)}, that is, we have to show that for the span ${\rm sp}(\beta, \alpha)=\mathbb{Z}_n$ holds. (Here $n = (d+1)(d+4) +1$.)

For differences $d+2$, $1$ and permutations
$\beta = (0, d, b_2,\ldots, b_{d})$, $\alpha= (0, a_1, \ldots, a_{d})$, we have for the span of $\beta$ and $\alpha$:
\begin{align*}
{\rm sp}(\beta, \alpha) = \{\pm ((q+1)(d+2) - (p+1)) \, : \, 0\leq p, q\leq d, \beta(p) + \alpha(q) \geq d+1\}.
\end{align*}
We show that ${\rm sp}(\beta, \alpha)$ and ${\rm sp}((0,d,\ldots,1), (0,d,\ldots,1))$ are equal for any  pair of permutations $(\beta, \alpha)$ satisfying the conditions in the  statement of the theorem. The theorem statement then follows from the fact  that the pair $((0,d,\ldots,1), (0,d,\ldots,1))$ gives rise to one of the two families from \cite{bdns} (see \eqref{eq:bdns}). In steps (a), \dots, (f) in the proof of Lemma 4.4 in \cite{bdns} it was shown that condition {\em (i)} is satisfied in this case.

\medskip

\noindent {\it Claim:}
Let $\beta = (0, d, b_2,\ldots, b_{d})$ and $\alpha= (0, a_1, \ldots, a_{d})$ be
such that $b_{i+1}+b_{d+1-i}=d$ and $a_i+a_{d+1-i}=d+1$ for all $i\geq 1$.
Then ${\rm sp}(\beta, \alpha)$ does not change when  modifying the vector $\alpha$ by $a_i\leftarrow a_i-1$ and  $a_{d+1-i} \leftarrow a_{d+1-i}+1$.

\medskip

The treetype tuple (see \eqref{eq:treetype}) that we lose by decreasing $a_i$ by one is $(i+1, j+1)$ where $a_i+b_j=d+1$. Similarly the treetype tuple we gain by increasing $a_{d+1-i}$ by one is $(d+2-i, s+1)$ where $a_{d+1-i}+b_s=d$ (it becomes $d+1$ after the increment). Summing, and using $a_i+a_{d+1-i}=d+1$, we get $b_j+b_s=d$. Thus $j+s=d+2$ (note that we are  using the fact that $\beta$ is a permutation). The corresponding differences lost and gained are $\pm d_{lost}$ and $\pm d_{gained}$ where
\begin{align}
    d_{lost} & = (d+2)(j+1) - (i+1), \nonumber \\
    d_{gained} & = (d+2)(s+1) - (d+2-i).
  \end{align}
From the above, we get $d_{lost}+d_{gained} = (d+2)(j+s+1)-1 = (d+2)(d+3)-1 = n = 0\in \mathbb{Z}_n$. So, $\{\pm d_{lost}\} = \{\mp d_{gained}\} \subseteq \mathbb{Z}_n$. Thus the span remains invariant under this modification.

\smallskip

Similarly it can be shown that if $\beta$ is modified analogously, the span
remains invariant as well. Now we transform an arbitrary $(\beta, \alpha)$ to that of the known example by: (i) keeping $\beta$ fixed and changing $\alpha$ to the decreasing permutation by a sequence of moves of $\alpha$ and then (ii) keeping $\alpha$ fixed and changing $\beta$ to the decreasing permutation by a sequence of moves on $\beta$. This proves the lemma.
\end{proof}

\begin{lemma} \label{lem:non-orientable}
Let $d$ be odd and let ${\mathcal D} = \{(\beta,\alpha)\}$ be
a $1$-deck of permutations as in the statement of Lemma~$\ref{lem:family}$. Then $\partial {\mathcal K}(G, {\mathcal D}T)$ is non-orientable.
\end{lemma}

\begin{proof}
By Lemma \ref{lem:family} and Theorem \ref{thm:perselect}, $\partial{\mathcal K}(G, {\mathcal D}T)$ is a member of ${\mathcal K}(d)$ and thus a triangulation of a manifold.

Let $C_i$, ${\mathcal D}T_j$, $\hat{u}$ be as in Subsections \ref{subsec:K(G,T)}, \ref{subsec:dualgr}, \ref{subsec:subtrees} and let $n=(d+1)(d+4)+1$.
Then $C_1$ is the $n$-cycle $v_{d+1}(0)\mbox{-}v_{d+1}(1)\mbox{-} \cdots \mbox{-} v_{d+1}(n-1)\mbox{-}v_{d+1}(0)$.

From the definition of ${\mathcal D}T_j$, it follows that the path $v_{d+1}(j)\mbox{-}v_{d+1}(j+ 1)\mbox{-} \cdots \mbox{-} v_{d+1}(j+(d+1))$ is a part (subgraph) of the tree ${\mathcal D}T_j$ for $0\leq j\leq n-1$. Thus,
\begin{align*}
\hat{v}_{d+1}(j) & := \{i \, : \, v_{d+1}(j)\in {\mathcal D}T_i\} = \{j, j-1, \dots, j-(d+1)\}.
\end{align*}
Hence $C_1$ is the dual graph of the $(d+1)$-dimensional (bounded) complex $Y_1$ given by
\begin{align*}
  Y_1 & = \{\hat{v}_{d+1}(j) \, : \, j \in \mathbb{Z}_n\} = \{\{j, j+1, \dots, j+d+1\} \, : \, j \in \mathbb{Z}_n\}.
\end{align*}
Let $X_n^d = \partial Y_1$. Since both $n$ and $d$ are odd, $X_n^d$ is non-orientable by \cite[Lemma 5.1]{bdns}. (In fact, $X^d_n$ triangulates the non-orientable $S^{d-1}$-bundle $\TPSSDS$ over $S^1$.)

From the construction it follows that $Y_1$ is a submanifold of ${\mathcal K}(G, {\mathcal D}T)$ and $|\partial{\mathcal K}(G, {\mathcal D}T)|$ can be obtained from $|\partial Y_1|$ by attaching handles. (This also follows from the topology of $|\partial{\mathcal K}(G, {\mathcal D}T)|$ described in Theorem \ref{thm:tight}.) Since $\partial Y_1 =X_n^d$ is non-orientable,
it follows that $\partial{\mathcal K}(G, {\mathcal D}T)$ is non-orientable.
\end{proof}

\begin{lemma}  \label{lem:orientable}
Let $d$ be even and let ${\mathcal D} = \{(\beta,\alpha)\}$ be
a $1$-deck of permutations as in the statement of Lemma~$\ref{lem:family}$. Then $\partial {\mathcal K}(G, {\mathcal D}T)$ is orientable.
\end{lemma}

\begin{proof}
By Lemma \ref{lem:family} and Theorem \ref{thm:perselect}, $\partial{\mathcal K}(G, {\mathcal D}T)$ is a member of ${\mathcal K}(d)$ and thus a triangulation of a manifold. It suffices to show that ${\mathcal K}(G, {\mathcal D}T)$ is orientable.

We prove that ${\mathcal K}(G, {\mathcal D}T)$ is orientable,
by describing an explicit coherent orientation (i.e., by orienting simplices such that the incidence numbers satisfy $[\sigma_1, \sigma_1\cap\sigma_2] = -[\sigma_2, \sigma_1\cap\sigma_2]$ for any two $(d+1)$-dimensional simplices $\sigma_1, \sigma_2$ connected by a common $d$-simplex).

\begin{figure}[ht]
\centering
\scalebox{0.8}{
\begin{tikzpicture}
\coordinate (v0j) at (0,0);
\coordinate (v0last) at (10.5,0);
\coordinate (vtop) at (0,6);
\coordinate (vtopright) at (10,6);

\draw [thick] (v0j) -- (v0last);
\draw [thick] (vtop) -- (vtopright);
\draw [thick] (v0j) --(0,2);
\draw [thick]  (vtop) -- (0,4);
\draw [dotted] (0,2) -- (0,4);

\foreach \i in {0,1,2,3,4,5,6,7,8,9}{
\draw (1.5+\i, 0) node {$\bullet$};
\draw (1+\i, 6) node {$\bullet$};
}

\foreach \i in {0,1,2,4,5,6}{
\draw (0,\i) node {$\bullet$};
}

\coordinate (v1j) at (0,1);
\coordinate (vdj) at (0,5);
\coordinate (vdj2) at (2,6);
\coordinate (v0l) at (5.5,0);
\coordinate (vtopl) at (5.0, 6);
\coordinate (v0tip) at (10.5, 1);
\coordinate (vtoptip) at (10, 5);
\coordinate (v1j2) at (2,1);
\coordinate (vtopmid) at (5,3);
\coordinate (v0mid) at (5.5,2);

\draw [thick] (vdj2) -- (2, 4) (2,2) -- (v1j2);
\draw [dotted] (2,4) -- (2,2);

\foreach \i in {5,4,2,1} {
	\draw (2,\i) node {$\bullet$};
}

\draw [thick] (vtopl) -- (vtopmid);	
\foreach \i in {5,4,3} {
	\draw (5,\i) node {$\bullet$};
}

\draw [thick] (v0l) -- (v0mid);
\foreach \i in {1,2} {
	\draw (5.5, \i) node {$\bullet$};
}

\draw [thick] (vtopright) -- (vtoptip)
				(v0last) -- (v0tip);

\draw (vtoptip) node {$\bullet$};
\draw (v0tip) node {$\bullet$};

\draw [dotted] (3,6) -- (3,3.5) (4,6) -- (4,4.5)
               (6,6) -- (6,4) (7,6) -- (7,5.5) (8,6) -- (8,4)
			   (9,6) -- (9,4.5);

\draw [dotted] (2.5,0) -- (2.5,0.5) (3.5, 0) -- (3.5, 1.5)
				(4.5, 0) -- (4.5,4) (6.5, 0) -- (6.5, 3)
				(7.5, 0) -- (7.5, 2.5) (8.5, 0) -- (8.5, 1.5)
				(9.5, 0) -- (9.5, 4.5);

\draw (v0j) node[left] {$v_0(j)$};
\draw (v1j) node[left] {$v_1(j)$};
\draw (vdj) node[left] {$v_d(j)$};
\draw (vtop) node[left] {$v_{d+1}(j)$};
\draw (v1j2) node[right] {$v_1(j+2)$};
\draw (vdj2) node[above right] {$v_{d+1}(j+2)$};

\draw (vtopl) node[above right] {$v_{d+1}(j+\ell)$};
\draw (v0l) node[below right] {$v_0(j+\ell(d+2))$};
\draw (vtopmid) node[right] {$v_{d+1-b_{\ell-1}}(j+\ell)$};
\draw (v0mid) node[right] {$v_{a_{\ell-1}}(j+\ell(d+2))$};
\draw (v0last) node[below right] {$v_0(j+(d+1)(d+2))$};
\draw (vtopright) node[above right] {$v_{d+1}(j+d+1)$};

\end{tikzpicture}
}
\caption{Tree $D{\cal T}_j$ in $G(d,1;d+2,1)$}
\label{fig:tree2}
\end{figure}

Observe that it follows from $d$ even that $Y_1$ (with dual graph $C_1$), as defined in the proof of Lemma \ref{lem:non-orientable}, is orientable (see \cite[Lemma 5.1]{bdns}).
We choose the positively oriented simplices as $\varepsilon \langle j, j-1, \dots, j-d-1\rangle$, where $\varepsilon = \pm 1$ is equal for all $j$. Similarly, the cycle $C_0$ yields an orientable manifold $Y_0$ with analogously oriented simplices.

From Definition 3.1, we have (see Figure \ref{fig:tree2})
\begin{align*}
{\mathcal D}T_j = &\,\, \left [\,v_{0}(j)\mbox{-}v_{0}(j+ (d+2))\mbox{-} \cdots \mbox{-} v_{0}(j+(d+1)(d+2)) \,\right ] \\
& \cup \,\, \left [\,v_{d+1}(j)\mbox{-} \cdots \mbox{-} v_{d+1}(j+(d+1))\, \right ] \\
& \cup \,\, P_j \\
& \cup \,\, \left [\bigcup _{1\leq \ell\leq d+1}\Big (\mbox{outward path of length } \alpha(\ell-1) \mbox{ at vertex } v_0(j+\ell(d+2)) \Big ) \right ] \\
& \cup \,\, \left [\bigcup _{1\leq \ell\leq d+1}\Big (\mbox{inward path of length } \beta(\ell-1) \mbox{ at vertex } v_{d+1}(j+\ell)\Big ) \right ].
\end{align*}
Hence we have $v_p(j+\ell(d+2)) \in {\mathcal D}T_j$ if and only if
$p\leq \alpha(\ell -1)$ and it follows from the cyclic symmetry that
$v_p(j) \in {\mathcal D}T_{j-\ell(d+2)}$ if and only if $p \leq \alpha(\ell -1)$. Similarly $v_p(j+\ell) \in {\mathcal D}T_j$ and
$v_p(j) \in {\mathcal D}T_{j-\ell}$ if and only if
$p\geq d+1-\beta(\ell -1)$. Thus we can write
\begin{align*}
  \hat{v}_p(j) & = \{m \, : \, v_p(j) \in {\mathcal D}T_m\} \\
   & = \{j-\ell(d+2) \, : \, p \leq \alpha(\ell-1)\} \,\cup \,\{j\}\, \cup\, \{j-\ell \, : \, p\geq d+1-\beta(\ell -1)\} \\
   & = \{j-(\ell+1)(d+2) : \ell \in \alpha^{-1}(\{p, \dots, d\})\}\, \cup \,\{j\} \\
& \qquad \cup \, \{j-(\ell+1) : \ell \in \beta^{-1}(\{d+1-p, \dots, d\})\} \nonumber \\
   & = \{j-(t_d+1)(d+2), \dots, j-(t_p+1)(d+2), j, j-(s_{d+1-p}+1), \dots, j - (s_d+1)\},
\end{align*}
where $s_i= \beta^{-1}(i)$, $t_i= \alpha^{-1}(i)$, $0\leq i\leq d$,
and $0\leq p\leq d+1$. Note that in the special cases
$p \in \{ 0, d+1 \}$ we have
\begin{align*}
  \hat{v}_{d+1}(j) & =  \{j, j-(s_{0}+1), j-(s_{1}+1), \dots, j - (s_d+1)\} \\
  & = \{j, j-1, \dots, j-(d+1)\}, \mbox{ and }\\
  \hat{v}_0(j) & = \{j-(t_{d}+1)(d+2), j-(t_{d-1}+1)(d+2),\dots, j - (t_0+1)(d+2), j\} \\
  & = \{j-(d+1)(d+2), j-d(d+2), \dots, j-(d+2),j\}.
\end{align*}

We assign a positive orientation to the ordered $(d+1)$-simplices
\begin{align*}\label{eq:orient}
  + \hat{v}_p(j) & = \langle j-(t_d+1)(d+2), \dots, j-(t_p+1)(d+2), j, j-(s_{d+1-p}+1), \dots, j - (s_d+1)\rangle.
\end{align*}
For the simplices in $Y_0$ and $Y_1$ this is equivalent to the
following signs with respect to the standard ordering of
$\hat{v}_{0}$ and $\hat{v}_{d+1}$:
\begin{align*}
  + \hat{v}_{d+1}(j) & =  \langle j, j-(s_{0}+1), j-(s_{1}+1), \dots, j - (s_d+1)\rangle \\
  & = \varepsilon(\beta)  \langle j, j-1, \dots, j-(d+1)\rangle,\\
  + \hat{v}_0(j) & = \langle j-(t_{d}+1)(d+2), \dots, j - (t_1+1)(d+2), j - (t_0+1)(d+2), j\rangle  \\
  &= \varepsilon(\alpha) \langle j-(d+1)(d+2), \dots, j-2(d+2), j-(d+2), j\rangle,
\end{align*}
where $\varepsilon(\beta)$ and $\varepsilon(\alpha)$ are the signs
of the permutations $\beta$ and $\alpha$ (note that all
$(d+1)$-simplices in $Y_0$ and $Y_1$ share the same sign).

\smallskip

It remains to show that the orientation chosen above is coherent for
all $d$-simplices of ${\mathcal K}(G, {\mathcal D}T)$.

Each $d$-simplex corresponds to one edge of $G$. There are three
types of edges in $G$. The ones in the inner cycle $C_0$, the
ones in the outer cycle $C_1$, and the ones in the vertical paths
$P_j$, $0 \leq j \leq n$.

We write
\begin{align*}
\gamma(j) := & \hat{v}_{d+1}(j) \cap \hat{v}_{d+1}(j+1) && \mbox{ for edges in } C_1, \\
\mu(j) := & \hat{v}_{0}(j) \cap \hat{v}_{0}(j+1) && \mbox{ for edges in } C_0, \mbox{ and} \\
\xi_p(j) := & \hat{v}_{p}(j) \cap \hat{v}_{p+1}(j),\, 0\leq p\leq d, && \mbox{ for edges in } P_j, \,j\in \mathbb{Z}_n.
\end{align*}
Then,
\begin{align*}
[\hat{v}_{d+1}(j), \gamma(j)] &\cdot [\hat{v}_{d+1}(j+1), \gamma(j)]  \\
 & = \varepsilon(\beta)^2 \cdot [\langle j, j-1, \dots, j-d-1\rangle, \gamma(j)]\cdot [\langle j+1, j, \dots, j-d\rangle, \gamma(j)] \\
& = \varepsilon(\beta)^2 \cdot (-1) = -1.
\end{align*}
An analogous computation yields $[\hat{v}_{0}(j), \mu(j)] = - [\hat{v}_{0}(j+1), \mu(j)]$.

For the vertical paths assume without loss of generality that
\begin{align*}
+\xi_p(j) &= \langle j-(t_d+1)(d+2), \dots, j-(t_{p+1}+1)(d+2), j, j-(s_{d+1-p}+1), \dots, j - (s_d+1) \rangle.
\end{align*}
Then
\begin{align*}
[\hat{v}_p(j), \xi_p(j)] = & (-1)^{d-p}
= -(-1)^{d-p+1} = - [\hat{v}_{p+1}(j), \xi_p(j)].
\end{align*}
It follows that the chosen orientation is coherent.
\end{proof}

\begin{theorem} \label{thm:family}
For $d\geq 2$, there are at least $2^{d-1}\cdot\lfloor d/2 \rfloor!\cdot \lfloor
(d-1)/2 \rfloor!$ non-isomorphic tight triangulations of $(S^{d-1}\times S^1)^{\# m}$,
$d$ even, and $(\TPSSDS)^{\# m}$, $d$ odd, where $m = \binom{(d+1)(d+3)+1}{2}/\binom{d+2}{2}$.
\end{theorem}

\begin{proof}
Consider the family $\{\partial{\cal K}(G(d,1; d+2, 1), {\cal D}T) \, : {\mathcal D} \}$ as in Lemma \ref{lem:family}. A simple count of all permutations $(\beta, \alpha)$
satisfying the pre-conditions of Lemma \ref{lem:family} shows that the number of members in this family is $2^{d-1}\cdot\lfloor d/2 \rfloor!\cdot \lfloor (d-1)/2 \rfloor!$.

By Lemmas \ref{lem:non-orientable} and \ref{lem:orientable}, for each 1-deck ${\mathcal D}=\{(\beta, \alpha)\}$ of permutations as in the statement of Lemma \ref{lem:family},
$\partial{\cal K}(G(d, 1; d+2, 1), {\cal D}T)$ is orientable if $d$ is even and non-orientable if $d$ is odd. The result now follows from Theorem \ref{thm:tight}.
\end{proof}

\subsection{Further sporadic examples} \label{sec:sp_eg}

In this section we summarise the findings of the algorithm in dimensions
three, four, and five for small values of $k$. (Since all neighbourly
triangulations of surfaces are known to be tight, we do not focus on
dimension two in this section. Nonetheless, the output of the algorithm for $d=2$ and
$k \leq 5$ is summarised in the first table below.) In each case the algorithm
terminated, meaning that all possible configurations were exhaustively tested.
Higher values of $d$ or $k$ turned out to have infeasible running times.

A selection of the examples are listed in Section \ref{sec:summary}
in terms of their corresponding $(d+1)$-dimensional handlebodies
by providing $m$-vectors $(m_0, m_1,  \ldots , m_{k-1}, 1)$, $k$-decks of
permutations, treetypes, and a system of
orbit representatives of the $(d+1)$-dimensional handlebody $M^{d+1}_{k, i}$ in
$\overline{\cal K}(d+1)$ under the vertex transitive $\mathbb{Z}_n$-symmetry
for $3\leq d\leq 5$. These are bounded tight manifolds. Their boundaries are
tight closed manifolds $\partial M^{d+1}_{k, i}$ in ${\cal K}(d)$ for $3\leq d\leq 5$.
The full list of examples is given in \cite{examples}. The complexes will also
be included in the next update of the {\em GAP}-package {\em simpcomp} \cite{simpcompISSAC,simpcomp}.

As an illustration
we discuss one new solitary example with $(d,k,n)=(3,2,49)$, which we denote by
$M^{d+1}_{k,i} = M^4_{2,1}$, in greater detail. For this example, we have
$(m_0,m_1,m_2)=(41,20,1)$ (see appendix). Also, we have the following
treetypes:
\begin{align*}
\Delta(\sigma_1, \tau_1) &=
\{ ( 1, 4 ), ( 3, 2 ), ( 3, 4 ), ( 4, 2 ), ( 4, 3 ), ( 4, 4 ) \}, \\
\Delta(\sigma_2,\tau_2) &=
\{ ( 1, 3 ), ( 2, 3 ), ( 2, 4 ), ( 4, 1 ), ( 4, 3 ), ( 4, 4 ) \},
\end{align*}
where ${\cal D}=\{(\sigma_1,\tau_1),(\sigma_2,\tau_2)\}$ is the $2$-deck of
permutations of $\{0,1,2,3\}$. It can be seen that we must have,
\begin{align*}
\sigma_1 = (1,2,0,3), \tau_1=(1,0,3,2); \sigma_2 = (1,0,2,3), \tau_2=(0,2,1,3).
\end{align*}

Figure \ref{fig:tree} shows the tree ${\cal D}T_0$ as an induced subgraph of
the graph $G(3,2;41,20,1)$.
Orbit representatives of the facets of
$M^4_{2,1}={\cal K}(G(3,2;41,20,1), {\cal D}T)$
under the group action of $\mathbb{Z}_{49} = \langle (0,1,\cdots,48) \rangle$
are given by:
\begin{align*}
    &\langle 0, 1, 2, 3, 4 \rangle,    \langle 0, 1, 3, 4, 31\rangle,  \langle 0, 1, 28, 37, 46\rangle, \langle 0, 4, 11, 31, 40\rangle,  \langle 0, 9, 18, 29, 38\rangle,\\
    & \quad \langle 0, 6, 15, 24, 41\rangle, \langle 0, 6, 16, 24, 41\rangle, \langle 0, 8, 16, 24, 32\rangle,   \langle 0, 5, 11, 20, 29 \rangle.
\end{align*}

\subsection{Summary of known examples} \label{sec:summary}

In this section we present a brief summary of the tight triangulations in Section~\ref{ssec:family}, Section~\ref{sec:sp_eg} and the literature. The full list of new tight triangulations is given in \cite{examples}, and includes the examples presented in Tables~\ref{tab:4_1} to \ref{tab:6_1}.

\bigskip
\medskip

\noindent
Neighbourly triangulations of surfaces generated by the algorithm:

\begin{center}
\begin{tabular}{|c|c|c|c|}
\hline
$k$ & $n$ & \#(solutions) & Remarks \\
\hline
0 & 7 & 1 & M\"obius' torus \\
\hline
1 & 19 & 2 & Both from \cite{bdns} \\
\hline
2 & 31 & 1 & New \\
\hline
3 & 43 & 4 & \ditto \\
\hline
4 & 55 & 1 & \ditto \\
\hline
5 & 67 & 56 & \ditto \\
\hline
\end{tabular}
\end{center}

\bigskip

\noindent
Examples of tight three-manifolds:

\begin{center}
\begin{tabular}{|c|c|c|c|}
\hline
$k$ & $n$ & \#(solutions) & Remarks \\
\hline
0 & 9 & 1 & Walkup \cite{wa}, K\"uhnel \cite{PSM} \\ 
\hline
1 & 29 & 6 & 4 in  Section~\ref{ssec:family}, including 2 in \cite{bdns}, and 2 in \cite{examples} \\
\hline
2 & 49 & 1 & \cite{examples}, also in Table~\ref{tab:4_1} \\
\hline
3 & 69 & 15 & \cite{examples}, including one in Table~\ref{tab:4_1} \\
\hline
4 & 89 & 41 & \ditto \\
\hline
5 & 109 & 72 & \ditto \\
\hline
\end{tabular}
\end{center}

\bigskip
\noindent
Examples of tight four-manifolds:

\begin{center}
\begin{tabular}{|c|c|c|c|}
\hline
$k$ & $n$ & \#(solutions) & Remarks \\
\hline
0 & 11 & 1 & K\"uhnel \cite{PSM} \\
\hline
1 & 41 & 16 &  16 in Section~\ref{ssec:family}, including 2 in \cite{bdns} \\
\hline
2 & 71 & 2 & \cite{examples}, including one in Table~\ref{tab:5_2} \\
\hline
3 & 101 & 2 & \ditto \\
\hline
\end{tabular}
\end{center}

\bigskip

\noindent
Examples of tight five-manifolds:

\begin{center}
\begin{tabular}{|c|c|c|c|}
\hline
$k$ & $n$ & \#(solutions) & Remarks \\
\hline
0 & 13 & 1 & K\"uhnel \cite{PSM} \\
\hline
1 & 55 & 90 &  64 in Section~\ref{ssec:family}, including 2 in \cite{bdns}, and 26 in \cite{examples} \\
\hline
2 & 97 & 13 & \cite{examples}, including one in Table~\ref{tab:5_2} \\
\hline
\end{tabular}
\end{center}

\begin{longtable}{|@{\hspace{0.05cm}}l@{\hspace{0.05cm}}|@{}l@{}|}
\caption{Some tight and minimal 4-dimensional handlebodies} \label{tab:4_1} \\

\endfirsthead

\multicolumn{2}{l}{ \tablename \,\,\thetable{} -- continued from previous page}
\\ \hline \endhead
\hline
\multicolumn{2}{r}{{continued on next page --}}
\\ \endfoot

\endlastfoot\hline\hline
{\scriptsize $M^{4}_{1,5}$}&{\scriptsize
\begin{tabular}{ll}
 $m$-vector:&
$\begin{array}{l}
(12,1) \end{array}$ \end{tabular}} \\
\hline
{\scriptsize Perms.:}&${\scriptsize \begin{array}{ll}
\sigma_{1} = (2,0,1,3),&\tau_{1} = (1,2,0,3)
\end{array}}$ \\
\hline
{\scriptsize Treetypes:}& ${\scriptsize \begin{array}{l}
\Delta (\sigma_{1},\tau_{1}) = \{
(1,2),
(1,4),
(3,4),
(4,1),
(4,2),
(4,4)\}
\end{array}}$ \\
\hline
{\scriptsize Orb. reps.:} & {\scriptsize $\begin{array}{llllll}
\langle 0, 1, 2, 3, 4\rangle,
&\langle 0, 1, 3, 4, 19\rangle,
&\langle 0, 1, 4, 19, 24\rangle,
&\langle 0, 4, 12, 19, 24\rangle,
&\langle 0, 5, 10, 17, 22\rangle&
\end{array}$} \\
\hline \hline
{\scriptsize $M^{4}_{2,1}$}&{\scriptsize
\begin{tabular}{ll}
 $m$-vector:&
$\begin{array}{l}
(41,20,1) \end{array}$ \end{tabular}} \\
\hline
{\scriptsize Perms.:}&${\scriptsize \begin{array}{lll}
\sigma_{2} = (1,0,2,3),&\tau_{2} = (0,2,1,3);\\
\sigma_{1} = (1,2,0,3),&\tau_{1} = (1,0,3,2)
\end{array}}$ \\
\hline
{\scriptsize Treetypes:}& ${\scriptsize \begin{array}{l}
\Delta (\sigma_{2},\tau_{2}) = \{
(1,4),
(3,2),
(3,4),
(4,2),
(4,3),
(4,4)\},\\
\Delta (\sigma_{1},\tau_{1}) = \{
(1,3),
(2,3),
(2,4),
(4,1),
(4,3),
(4,4)\}
\end{array}}$ \\
\hline
{\scriptsize Orb. reps.:} & {\scriptsize $\begin{array}{l@{}l@{}l@{}l@{}l}
\langle 0, 1, 2, 3, 4\rangle,\,\,
&\langle 0, 1, 3, 4, 31\rangle,\,\,
&\langle 0, 1, 28, 37, 46\rangle,\,\,
&\langle 0, 4, 11, 31, 40\rangle,\,\,
&\langle 0, 9, 18, 29, 38\rangle,
\\
\langle 0, 6, 15, 24, 41\rangle,\,\,
&\langle 0, 6, 16, 24, 41\rangle,\,\,
&\langle 0, 8, 16, 24, 32\rangle,\,\,
&\langle 0, 5, 11, 20, 29\rangle&
\end{array}$} \\
\hline \hline
{\scriptsize $M^{4}_{3,1}$}&{\scriptsize
\begin{tabular}{ll}
 $m$-vector:&
$\begin{array}{l}
(64,25,8,1) \end{array}$ \end{tabular}} \\
\hline
{\scriptsize Perms.:}&${\scriptsize \begin{array}{ll}
\sigma_{3} = (0,2,3,1),&\tau_{3} = (0,3,2,1);\\
\sigma_{2} = (3,2,0,1),&\tau_{2} = (3,2,0,1);\\
\sigma_{1} = (1,0,2,3),&\tau_{1} = (3,1,0,2)
\end{array}}$ \\
\hline
{\scriptsize Treetypes:}& ${\scriptsize \begin{array}{l}
\Delta (\sigma_{3},\tau_{3}) = \{
(2,2),
(2,3),
(3,2),
(3,3),
(3,4),
(4,2)\},\\
\Delta (\sigma_{2},\tau_{2}) = \{
(1,1),
(1,2),
(1,4),
(2,1),
(2,2),
(4,1)\},\\
\Delta (\sigma_{1},\tau_{1}) = \{
(1,1),
(3,1),
(3,4),
(4,1),
(4,2),
(4,4)\}
\end{array}}$ \\
\hline
{\scriptsize Orb. reps.:} & {\scriptsize $\begin{array}{l@{}l@{}@{}l@{}l@{}l@{}}
\langle 0, 1, 2, 3, 4\rangle,\,\,
&\langle 0, 1, 2, 14, 67\rangle,\,\,
&\langle 0, 1, 14, 22, 67\rangle,\,\,
&\langle 0, 3, 16, 24, 32\rangle,\,\,
&\langle 0, 6, 25, 31, 50\rangle,
\\
\langle 0, 5, 11, 30, 36\rangle,\,\,
&\langle 0, 5, 11, 36, 54\rangle,\,\,
&\langle 0, 8, 16, 24, 32\rangle,\,\,
&\langle 0, 7, 44, 52, 60\rangle,\,\,
&\langle 0, 8, 16, 25, 50\rangle,
\\
\langle 0, 6, 25, 44, 52\rangle,\,\,
&\langle 0, 5, 10, 41, 59\rangle,\,\,
&\langle 0, 5, 10, 15, 20\rangle&&
\end{array}$} \\
\hline \hline
{\scriptsize $M^{4}_{4,1}$}&{\scriptsize
\begin{tabular}{ll}
 $m$-vector:&
$\begin{array}{l}
(36,78,20,7,1) \end{array}$ \end{tabular}} \\
\hline
{\scriptsize Perms.:}&${\scriptsize \begin{array}{ll}
\sigma_{4} = (3,0,1,2),&\tau_{4} = (1,2,0,3);\\
\sigma_{3} = (0,2,1,3),&\tau_{3} = (1,0,3,2);\\
\sigma_{2} = (0,3,2,1),&\tau_{2} = (2,0,1,3);\\
\sigma_{1} = (2,3,0,1),&\tau_{1} = (2,0,3,1)
\end{array}}$ \\
\hline
{\scriptsize Treetypes:}& ${\scriptsize \begin{array}{l}
\Delta (\sigma_{4},\tau_{4}) = \{
(1,1),
(1,2),
(1,4),
(3,4),
(4,2),
(4,4)\},\\
\Delta (\sigma_{3},\tau_{3}) = \{
(2,3),
(2,4),
(3,3),
(4,1),
(4,3),
(4,4)\},\\
\Delta (\sigma_{2},\tau_{2}) = \{
(2,1),
(2,3),
(2,4),
(3,1),
(3,4),
(4,4)\},\\
\Delta (\sigma_{1},\tau_{1}) = \{
(1,1),
(1,3),
(2,1),
(2,3),
(2,4),
(4,3)\}
\end{array}}$ \\
\hline
{\scriptsize Orb. reps.:} & {\scriptsize $\begin{array}{l@{}l@{}l@{}l@{}l}
\langle 0, 1, 2, 3, 4\rangle,\,\,
&\langle 0, 1, 3, 4, 28\rangle,\,\,
&\langle 0, 1, 4, 14, 28\rangle,\,\,
&\langle 0, 1, 7, 14, 28\rangle,\,\,
&\langle 0, 7, 14, 21, 28\rangle,
\\
\langle 0, 7, 14, 46, 75\rangle,\,\,
&\langle 0, 9, 23, 37, 69\rangle,\,\,
&\langle 0, 12, 34, 53, 70\rangle,\,\,
&\langle 0, 17, 34, 53, 70\rangle,\,\,
&\langle 0, 11, 22, 41, 58\rangle,
\\
\langle 0, 11, 22, 41, 67\rangle,\,\,
&\langle 0, 8, 40, 60, 69\rangle,\,\,
&\langle 0, 9, 29, 49, 69\rangle,\,\,
&\langle 0, 5, 16, 38, 49\rangle,\,\,
&\langle 0, 5, 20, 38, 49\rangle,
\\
\langle 0, 5, 20, 40, 49\rangle,
&\langle 0, 11, 22, 33, 44\rangle&&&
\end{array}$} \\
\hline \hline
{\scriptsize $M^{4}_{5,1}$}&{\scriptsize
\begin{tabular}{ll}
 $m$-vector:&
$\begin{array}{l}
(91,26,86,75,8,1) \end{array}$ \end{tabular}} \\
\hline
{\scriptsize Perms.:}&${\scriptsize \begin{array}{ll}
\sigma_{5} = (0,3,1,2),&\tau_{5} = (1,0,2,3);\\
\sigma_{4} = (3,2,0,1),&\tau_{4} = (1,2,0,3);\\
\sigma_{3} = (3,0,2,1),&\tau_{3} = (1,3,0,2);\\
\sigma_{2} = (1,0,3,2),&\tau_{2} = (3,0,2,1);\\
\sigma_{1} = (2,3,0,1),&\tau_{1} = (3,2,0,1)
\end{array}}$ \\
\hline
{\scriptsize Tree-\linebreak types:}& ${\scriptsize \begin{array}{l}
\Delta (\sigma_{5},\tau_{5}) = \{
(2,1),
(2,3),
(2,4),
(3,4),
(4,3),
(4,4)\},\\
\Delta (\sigma_{4},\tau_{4}) = \{
(1,1),
(1,2),
(1,4),
(2,2),
(2,4),
(4,4)\},\\
\Delta (\sigma_{3},\tau_{3}) = \{
(1,1),
(1,2),
(1,4),
(3,2),
(3,4),
(4,2)\},\\
\Delta (\sigma_{2},\tau_{2}) = \{
(1,1),
(3,1),
(3,3),
(3,4),
(4,1),
(4,3)\},\\
\Delta (\sigma_{1},\tau_{1}) = \{
(1,1),
(1,2),
(2,1),
(2,2),
(2,4),
(4,1)\}
\end{array}}$ \\
\hline
{\scriptsize Orb. \linebreak reps.:} & {\scriptsize $\begin{array}{l@{}l@{}l@{}l@{}l}
\langle 0, 1, 2, 3, 4\rangle,\,\,
&\langle 0, 1, 2, 30, 107\rangle,\,\,
&\langle 0, 2, 4, 24, 32\rangle,\,\,
&\langle 0, 2, 8, 24, 32\rangle,\,\,
&\langle 0, 7, 17, 63, 75\rangle,
\\
\langle 0, 7, 34, 41, 75\rangle,\,\,
&\langle 0, 7, 34, 41, 97\rangle,\,\,
&\langle 0, 8, 16, 24, 32\rangle,\,\,
&\langle 0, 8, 16, 32, 82\rangle,\,\,
&\langle 0, 8, 16, 41, 82\rangle,
\\
\langle 0, 7, 34, 42, 75\rangle,\,\,
&\langle 0, 9, 23, 61, 92\rangle,\,\,
&\langle 0, 9, 23, 69, 92\rangle,\,\,
&\langle 0, 17, 40, 63, 86\rangle,\,\,
&\langle 0, 11, 34, 51, 97\rangle,
\\
\langle 0, 5, 31, 45, 83\rangle,\,\,
&\langle 0, 18, 36, 62, 88\rangle,\,\,
&\langle 0, 5, 31, 57, 83\rangle,\,\,
&\langle 0, 5, 31, 57, 96\rangle,\,\,
&\langle 0, 15, 36, 54, 72\rangle,
\\
\langle 0, 18, 36, 54, 72\rangle&&&&
\end{array}$} \\
\hline \hline
\end{longtable}

\bigskip
\bigskip
\begin{longtable}{|@{\hspace{0.05cm}}l@{\hspace{0.05cm}}|@{}l@{}|}
\caption{Some tight and minimal 5-dimensional handlebodies} \label{tab:5_2} \\

\endfirsthead

\multicolumn{2}{l}{ \tablename \,\,\thetable{}
 -- continued from previous page}
\\ \hline \endhead
\hline
\multicolumn{2}{r}{{continued on next page --}}
\\ \endfoot

\endlastfoot\hline\hline
{\scriptsize $M^{5}_{1,1}$}&{\scriptsize
\begin{tabular}{ll}
 $m$-vector:&
$\begin{array}{l}
(6,1) \end{array}$ \end{tabular}} \\
\hline
{\scriptsize Perms.:}&${\scriptsize \begin{array}{ll}
\sigma_{1} = (0,1,2,3,4),&\tau_{1} = (0,4,3,2,1)
\end{array}}$ \\
\hline
{\scriptsize Treetypes:}& ${\scriptsize \begin{array}{l}
\Delta (\sigma_{1},\tau_{1}) = \{
(2,2),
(3,2),
(3,3),
(4,2),
(4,3),
(4,4),
(5,2),
(5,3),
(5,4),
(5,5)\}
\end{array}}$ \\
\hline
{\scriptsize Orb. reps.:} & {\scriptsize $\begin{array}{l@{}l@{}l}
\langle 0, 1, 2, 3, 4, 5\rangle,\,\,
&\langle 0, 1, 2, 3, 10, 39\rangle,\,\,
&\langle 0, 1, 2, 9, 15, 38\rangle,
\\
\langle 0, 1, 8, 14, 20, 37\rangle,\,\,
&\langle 0, 5, 12, 18, 24, 30\rangle,\,\,
&\langle 0, 6, 12, 18, 24, 30\rangle
\end{array}$} \\
\hline \hline
{\scriptsize $M^{5}_{2,1}$}&{\scriptsize
\begin{tabular}{ll}
 $m$-vector:&
$\begin{array}{l}
(44,8,1) \end{array}$ \end{tabular}} \\
\hline
{\scriptsize Perms.:}&${\scriptsize \begin{array}{ll}
\sigma_{2} = (0,3,1,4,2),&\tau_{2} = (0,4,3,2,1);\\
\sigma_{1} = (1,2,0,3,4),&\tau_{1} = (0,2,4,1,3)
\end{array}}$ \\
\hline
{\scriptsize Treetypes:}& ${\scriptsize \begin{array}{l}
\Delta (\sigma_{2},\tau_{2}) = \{
(2,2),
(2,3),
(2,4),
(3,2),
(4,2),
(4,3),
(4,4),
(4,5),
(5,2),
(5,3)\},\\
\Delta (\sigma_{1},\tau_{1}) = \{
(1,3),
(2,3),
(2,5),
(4,2),
(4,3),
(4,5),
(5,2),
(5,3),
(5,4),
(5,5)\}
\end{array}}$ \\
\hline
{\scriptsize Orb. reps.:} & {\scriptsize $\begin{array}{l@{}l@{}l@{}l}
\langle 0, 1, 2, 3, 4, 5\rangle,\,\,
&\langle 0, 1, 2, 3, 14, 69\rangle,\,\,
&\langle 0, 1, 12, 20, 67, 69\rangle,\,\,
&\langle 0, 2, 4, 16, 24, 32\rangle,
\\
\langle 0, 4, 16, 24, 32, 40\rangle,\,\,
&\langle 0, 7, 16, 32, 40, 61\rangle,\,\,
&\langle 0, 7, 17, 32, 40, 61\rangle,\,\,
&\langle 0, 6, 27, 37, 44, 54\rangle,
\\
\langle 0, 7, 17, 34, 44, 61\rangle,\,\,
&\langle 0, 8, 16, 24, 32, 40\rangle,\,\,
&\langle 0, 8, 16, 32, 40, 61\rangle&
\end{array}$} \\
\hline \hline
{\scriptsize $M^{5}_{3,1}$}&{\scriptsize
\begin{tabular}{ll}
 $m$-vector:&
$\begin{array}{l}
(62,43,12,1) \end{array}$ \end{tabular}} \\
\hline
{\scriptsize Perms.:}&${\scriptsize \begin{array}{ll}
\sigma_{3} = (0,3,2,4,1),&\tau_{3} = (3,4,2,0,1);\\
\sigma_{2} = (3,0,2,1,4),&\tau_{2} = (2,4,0,3,1);\\
\sigma_{1} = (0,1,2,3,4),&\tau_{1} = (1,0,3,4,2)
\end{array}}$ \\
\hline
{\scriptsize Treetypes:}& ${\scriptsize \begin{array}{l}
\Delta (\sigma_{3},\tau_{3}) = \{
(2,1),
(2,2),
(2,3),
(3,1),
(3,2),
(4,1),
(4,2),
(4,3),
(4,5),
(5,2)\},\\
\Delta (\sigma_{2},\tau_{2}) = \{
(1,1),
(1,2),
(1,4),
(3,2),
(3,4),
(4,2),
(5,1),
(5,2),
(5,4),
(5,5)\},\\
\Delta (\sigma_{1},\tau_{1}) = \{
(2,4),
(3,3),
(3,4),
(4,3),
(4,4),
(4,5),
(5,1),
(5,3),
(5,4),
(5,5)\}
\end{array}}$ \\
\hline
{\scriptsize Orb. reps.:} & {\scriptsize $\begin{array}{l@{}l@{}l@{}l@{\hspace{0.05cm}}}
\langle 0, 1, 2, 3, 4, 5\rangle,\,\,
&\langle 0, 1, 2, 3, 22, 99\rangle,\,\,
&\langle 0, 1, 2, 10, 22, 99\rangle,\,\,
&\langle 0, 2, 4, 12, 24, 36\rangle,
\\
\langle 0, 4, 12, 24, 36, 60\rangle,\,\,
&\langle 0, 6, 39, 55, 78, 94\rangle,\,\,
&\langle 0, 6, 39, 64, 78, 94\rangle,\,\,
&\langle 0, 7, 23, 46, 62, 85\rangle,
\\
\langle 0, 12, 24, 36, 48, 60\rangle,\,\,
&\langle 0, 12, 24, 50, 65, 77\rangle,\,\,
&\langle 0, 11, 26, 41, 53, 77\rangle,\,\,
&\langle 0, 11, 26, 41, 53, 84\rangle,
\\
\langle 0, 13, 28, 43, 71, 86\rangle,\,\,
&\langle 0, 13, 28, 46, 71, 85\rangle,\,\,
&\langle 0, 13, 28, 46, 71, 86\rangle,\,\,
&\langle 0, 11, 26, 41, 54, 84\rangle
\end{array}$} \\
\hline \hline
\end{longtable}

\begin{longtable}{|@{\hspace{0.05cm}}l@{\hspace{0.05cm}}|@{}l@{}|}
\caption{Some tight and minimal 6-dimensional handlebodies} \label{tab:6_1} \\

\endfirsthead

\multicolumn{2}{l}{ \tablename \,\,\thetable{}
 -- continued from previous page}
\\ \hline \endhead
\hline
\multicolumn{2}{r}{{continued on next page --}}
\\ \endfoot

\endlastfoot\hline\hline
{\scriptsize $M^{6}_{1,65}$}&{\scriptsize
\begin{tabular}{ll}
 $m$-vector:&
$\begin{array}{l}
(16,1) \end{array}$ \end{tabular}} \\
\hline
{\scriptsize Perms.:}&${\scriptsize \begin{array}{ll}
\sigma_{1} = (3,0,4,5,1,2),&\tau_{1} = (1,2,4,0,5,3)
\end{array}}$ \\
\hline
{\scriptsize Treetypes:}& ${\scriptsize \begin{array}{l@{}l@{}l}
\Delta (\sigma_{1},\tau_{1}) &\,\,= \{&
(1,3),
(1,5),
(1,6),
(3,2),
(3,3),
(3,5),
(3,6), \\ &&
(4,1),
(4,2),
(4,3),
(4,5),
(4,6),
(5,5),
(6,3),
(6,5)\}
\end{array}}$ \\
\hline
{\scriptsize Orb. reps.:} & {\scriptsize $\begin{array}{l@{}l@{}l@{}l}
\langle 0, 1, 2, 3, 4, 5, 6\rangle,\,\,
&\langle 0, 1, 2, 3, 22, 52, 53\rangle,\,\,
&\langle 0, 1, 3, 4, 6, 25, 48\rangle,\,\,
&\langle 0, 1, 3, 4, 25, 41, 48\rangle,
\\
\langle 0, 1, 22, 29, 38, 45, 52\rangle,\,\,
&\langle 0, 4, 16, 25, 32, 41, 48\rangle,\,\,
&\langle 0, 7, 14, 23, 30, 39, 46\rangle&
\end{array}$} \\
\hline \hline
{\scriptsize $M^{6}_{2,1}$}&{\scriptsize
\begin{tabular}{ll}
 $m$-vector:&
$\begin{array}{l}
(87,7,1) \end{array}$ \end{tabular}} \\
\hline
{\scriptsize Perms.:}&${\scriptsize \begin{array}{ll}
\sigma_{2} = (1,4,5,0,2,3),&\tau_{2} = (0,5,4,3,2,1);\\
\sigma_{1} = (2,5,4,3,0,1),&\tau_{1} = (5,4,2,3,0,1)
\end{array}}$ \\
\hline
{\scriptsize Treetypes:}& ${\scriptsize \begin{array}{l@{}l@{}l}
\Delta (\sigma_{2},\tau_{2}) &\,\,= \{&
(1,2),
(2,2),
(2,3),
(2,4),
(2,5),
(3,2),
(3,3), \\ &&
(3,4),
(3,5),
(3,6),
(5,2),
(5,3),
(6,2),
(6,3),
(6,4)\},\\
\Delta (\sigma_{1},\tau_{1}) &\,\,= \{&
(1,1),
(1,2),
(2,1),
(2,2),
(2,3),
(2,4),
(2,6), \\ &&
(3,1),
(3,2),
(3,3),
(3,4),
(4,1),
(4,2),
(4,4),
(6,1)\}
\end{array}}$ \\
\hline
{\scriptsize Orb. reps.:} & {\scriptsize $\begin{array}{l@{}l@{}l@{}l}
\langle 0, 1, 2, 3, 4, 5, 6\rangle,\,\,
&\langle 0, 1, 2, 3, 5, 6, 14\rangle,\,\,
&\langle 0, 1, 3, 4, 12, 19, 95\rangle,\,\,
&\langle 0, 1, 4, 12, 19, 26, 95\rangle,
\\
\langle 0, 1, 12, 19, 26, 33, 95\rangle,\,\,
&\langle 0, 3, 14, 21, 28, 35, 42\rangle,\,\,
&\langle 0, 7, 14, 21, 28, 35, 42\rangle,\,\,
&\langle 0, 7, 14, 21, 28, 42, 87\rangle,
\\
\langle 0, 7, 14, 21, 28, 77, 87\rangle,\,\,
&\langle 0, 7, 14, 43, 63, 73, 83\rangle,\,\,
&\langle 0, 7, 43, 53, 63, 73, 83\rangle,\,\,
&\langle 0, 10, 20, 30, 40, 54, 77\rangle,
\\
\langle 0, 10, 20, 30, 40, 50, 60\rangle&&&
\end{array}$} \\
\hline \hline
\end{longtable}

\end{document}